\documentclass[]{article}
\usepackage{amssymb,amsmath,amsthm,bm,citesort,setspace}
\usepackage{url,color,units}
\usepackage{pdfsync}
\input cyracc.def

\newcommand{\Z}{\mathbf{Z}}

\newcommand{\R}{\mathbf{R}}

\newcommand{\be}{\begin{equation}}
\newcommand{\ee}{\end{equation}}

\newcommand{\lip}{\text{\rm Lip}_\sigma }

\renewcommand{\P}{\mathrm{P}}
\newcommand{\E}{\mathrm{E}}
\newcommand{\F}{\mathcal{F}}

\newcommand{\1}{\boldsymbol{1}}
\renewcommand{\d}{{\rm d}}

\newcommand{\e}{{\rm e}}
\renewcommand{\leq}{\leqslant}
\renewcommand{\le}{\leqslant}
\renewcommand{\geq}{\geqslant}
\renewcommand{\ge}{\geqslant}

\author{Daniel Conus\\Lehigh University
\and Mathew Joseph\\University of Utah
\and Davar Khoshnevisan\\University of Utah
\and  Shang-Yuan Shiu\\Academica Sinica}
\title{ On the chaotic character of\\the stochastic heat equation, II\thanks{%
	Research supported in part by 
	the NSFs grant DMS-0747758 (M.J.) and DMS-1006903 (D.K.).}}
\date{November 20, 2011}
\newtheorem{stat}{Statement}[section]
\newtheorem{proposition}[stat]{Proposition}
\newtheorem{corollary}[stat]{Corollary}
\newtheorem{theorem}[stat]{Theorem}
\newtheorem{lemma}[stat]{Lemma}
\theoremstyle{definition} 

\newtheorem{remark}[stat]{Remark}

\newtheorem{example}[stat]{Example}

\numberwithin{equation}{section}
\begin{document}
\maketitle
\begin{abstract}
	Consider the stochastic heat equation 
	$\partial_t u = (\nicefrac\varkappa2)\Delta u+\sigma(u)\dot{F}$,
	where the solution $u:=u_t(x)$ is indexed by $(t\,,x)\in (0\,,\infty)\times\R^d$,
	 and $\dot{F}$ is a centered Gaussian noise that is white in time
	 and has spatially-correlated coordinates. We analyze the
	 large-$\|x\|$ fixed-$t$ behavior of the solution $u$ in different
	 regimes, thereby study the effect of noise on the solution in various
	 cases. Among other things, we show that if the spatial correlation
	 function $f$ of the noise is of Riesz type, that is $f(x)\propto \|x\|^{-\alpha}$,
	 then the ``fluctuation exponents''
	 of the solution are $\psi$ for the spatial variable and $2\psi-1$ for the time
	 variable, where $\psi:=2/(4-\alpha)$. Moreover, these exponent relations
	 hold as long as $\alpha\in(0\,,d\wedge 2)$; that is precisely when
	 Dalang's theory \cite{Dalang:99} implies the existence
	 of a solution to our stochastic PDE. These findings bolster  
	 earlier physical predictions \cite{KPZ,KZ}.\\


\end{abstract}

\section{Introduction}

Consider the  nonlinear stochastic heat equation,
$$
	\frac{\partial}{\partial t} u_t(x) = \frac{\varkappa}{2}(\Delta u_t)(x)
	+\sigma(u_t(x))\dot{F}_t(x),\eqno(\textnormal{SHE})
$$
where $\varkappa>0$ is a viscosity constant, $\sigma:\R\to\R$
is globally Lipschitz continuous, and  $\{\dot{F}_t(x)\}_{t>0,x\in\R^d}$ is 
a centered generalized Gaussian random field \cite[Chapter 2, \S2.7]{GV} 
with covariance measure
\begin{equation}
	\textnormal{Cov}\left( \dot{F}_t(x)\,,\dot{F}_s(y)\right)
	=\delta_0(t-s) f(x-y)
\end{equation}
of the convolution type. We also assume, mostly for the
sake of technical simplicity,  that
the initial function $u_0:\R^d\to\R$ is nonrandom, essentially bounded,
and measurable. In particular, we assume the following once 
and for all:
\begin{equation}
	\text{Throughout this paper,  we assume that 
	$\|u_0\|_{L^\infty(\R^d)}<\infty$,}
\end{equation}
and that the correlation function
$f$ is sufficiently nice that there exists a unique strong solution
to (\textnormal{SHE}); see the next section for the technical details.

Our first result (Theorem \ref{th:boundedness})
tells us that if the initial function $u_0$
decays at infinity faster than exponentially, then the solution 
$x\mapsto u_t(x)$ is typically globally bounded at all nonrandom times $t>0$.
The remainder of this paper is concerned with showing that if by contrast $u_0$
remains uniformly  away from zero, then the typical structure of
the random function $x\mapsto u_t(x)$ is quite different from the behavior
outlined in Theorem \ref{th:boundedness}. In particular,
our results show that the solution to \textnormal{(SHE)} depends in a
very sensitive way on the structure of the initial function $u_0$.
[This property explains the appearance of ``chaos'' in the title of the paper.]

Hereforth, we assume tacitly that $u_0$ is bounded uniformly away from 
zero and infinity. We now describe the remaining contributions of this
paper [valid for such choices of $u_0$].

Loosely speaking, $\dot{F}_t(x)$ is nothing 
but white noise in the time variable $t$, and has a homogenous spatial 
correlation function $f$ for its space variable $x$.  In a companion paper
\cite{CJK} we study \textnormal{(SHE)} in the case that $\dot{F}$ is replaced
with space-time
white noise; that is the case where we replace the covariance measure
with $\delta_0(t-s)\delta_0(x-y)$. In that case, the solution exists 
only when $d=1$ \cite{Dalang:99,PZ,Walsh}. Before we describe the
results of \cite{CJK}, let us introduce some notation.

Let $h,g:\R^d\to\R_+$ be two functions.
We write: (a) ``$h(x)\succ g(x)$'' when
$\limsup_{\|x\|\to\infty}[ h(x)/g(x)]$ is bounded below by a constant; 
(b) ``$h(x)\asymp g(x)$'' when 
$ h(x)\succ  g(x)$ and $ g(x) \succ  h(x)$ both hold; 
and finally (c)  ``$h(x)\stackrel{(\log)}{\approx} g(x)$''
means that $\log h(x) \asymp \log g(x)$.

Armed with this notation, we can describe some of the findings of
\cite{CJK} as follows:
\begin{enumerate}
\item[1.] 
	If $\sigma$ is bounded uniformly away from zero, then
	$u_t(x) \succ  \varkappa^{-\nicefrac{1}{12}} (\log\|x\|)^{%
	\nicefrac16}$
	a.s.\ for all times $t>0$, where the constant in ``$\succ $'' does
	not depend on $\varkappa$;
\item[2.] 
	If $\sigma$ is bounded uniformly away from zero and infinity,
	then $u_t(x)\asymp \varkappa^{-\nicefrac14}(\log\|x\|)^{%
	\nicefrac12}$ a.s.\ for all $t>0$, where the constant in ``$\asymp$''
	holds uniformly for all $\varkappa\geq\varkappa_0$ for every
	fixed $\varkappa_0>0$; and
\item[3.] 
	If $\sigma(z)=cz$ for some $c>0$---and \textnormal{(SHE)} is in that
	case called the ``parabolic
	Anderson model'' \cite{CM}---then
	\begin{equation}\label{3}
		u_t(x) \stackrel{(\log)}{\approx} \exp\left( \frac{(\log \|x\|)^\psi}{
		\varkappa^{2\psi-1}}\right),
	\end{equation}
	for $\psi=\nicefrac23$ and $2\psi-1=\nicefrac13$, 
	valid a.s.\ for all $t>0$.\footnote{%
	Even though the variable $x$ is one-dimensional here,
	we write ``$\|x\|$'' in place of ``$|x|$'' because
	we revisit \eqref{3}  in the next few
	paragraphs and consider the case that $x\in\R^d$
	for  $d\ge 1$. }
\end{enumerate}
	
Coupled with the results of \cite{FK:AIHP}, the preceding
facts show that the solution to the stochastic heat equation \textnormal{(SHE)}, driven by 
space-time white noise, depends sensitively on the choice of the initial
data. 

Let us emphasize that these findings [and the subsequent ones of the
present paper] are remarks about the effect of the noise on the solution
to the PDE \textnormal{(SHE)}. Indeed, it is easy to see that if $u_0(x)$ is identically
equal to one---this is permissible in the present setup---then the distribution
of $u_t(x)$ is independent of $x$. Therefore, the limiting behaviors described 
above cannot be detected by looking at the distribution of $u_t(x)$ alone
for a fixed $x$. Rather it is the correlation between $u_t(x)$ and $u_t(y)$
that plays an important role.
	
The goal of the present paper is to study the effect of disorder
on the ``intermittent'' behavior of the solution to \textnormal{(SHE)};
specifically, we consider spatially-homogeneous
correlation functions of the form
$f(x-y)$ that are fairly nice, and think of the viscosity coefficient
$\varkappa$ as small, but positive.  Dalang's theory
\cite{Dalang:99} can be used to show that the stochastic PDE
\textnormal{(SHE)} has a  solution in  all dimensions if $f(0)<\infty$;
and it turns out that typically the following are valid,
as $\|x\|\to\infty$:
\begin{enumerate}
\item[$1'$.] 
	If $\sigma$ is bounded uniformly away from zero, then
	$u_t(x) \succ  (\log\|x\|)^{\nicefrac14}$
	for all times $t>0$, uniformly for all $\varkappa>0$ small;
\item[$2'$.] 
	If $\sigma$ is bounded uniformly away from zero and infinity,
	then $u_t(x)\asymp (\log\|x\|)^{\nicefrac12}$ for 
	all $t>0$, uniformly for all $\varkappa>0$ small; and
\item[$3'$.] If $\sigma(z)=cz$ for some $c>0$
	[the parabolic Anderson model] then \eqref{3} holds
	with  $\psi=\nicefrac12$ and $2\psi-1=0$,
	for all $t>0$ and uniformly for all $\varkappa>0$ small.
\end{enumerate}
Thus, we find that for nice bounded correlation functions,
the level of disorder [as measured by  $\nicefrac1\varkappa$]
does not play a role in determining the asymptotic large-$\|x\|$ 
behavior of the solution,  whereas it does for $f(x-y)=\delta_0(x-y)$. 
In other words, $1'$, $2'$, and $3'$ are in sharp
contrast to $1$, $2$, and $3$  respectively.
This contrast
can be explained loosely as saying that when $f$ is nice, the model
is ``mean field''; see in particular the application of the typically-crude
inequality \eqref{eq:Jensen}, which is shown to be sharp in this context.

One can think of the viscosity coefficient $\varkappa$ as ``inverse time'' by 
making analogies with finite-dimensional diffusions. 
As such, \eqref{3} suggests a kind
of space-time scaling that is valid universally for many choices of initial
data $u_0$; interestingly enough this very scaling law [$\psi$ versus
$2\psi-1$] has been predicted in the physics literature \cite{KZ,KPZ}, 
and several parts of it have been proved rigorously
in recent works by Bal\'azs, Quastel, and Sepp\"al\"ainen \cite{BQS}
and Amir, Corwin, and Quastel \cite{ACQ} in a large-$t$ fixed-$x$ regime.

We mentioned that \eqref{3} holds for $\psi=\nicefrac23$ [space-time white noise]
and $\psi=\nicefrac12$ [$f$ nice and bounded]. In the last portion of
this paper we prove that there are models---for the correlation function
$f$ of the noise $\dot{F}$---that satisfy \eqref{3} for every 
$\psi\in(\nicefrac12\,,\nicefrac23)$ in dimension $d=1$ and
for every $\psi\in(\nicefrac12\,,1)$ in dimension $d\ge 2$.
It is possible that these results reinforce  the ``superuniversality'' predictions of Kardar
and Zhang \cite{KZ}.

We conclude the introduction by setting forth some notation that
will be used throughout, and consistently. 

Let $p_t(z)$ denote
the heat kernel for $(\nicefrac\varkappa2)\Delta$ on $\R^d$; that is,
\begin{equation}
	p_t(z) := 
	\frac{1}{\left(2\pi \varkappa t\right)^{d/2}}\exp\left(-\frac{
	\|z\|^2}{2\varkappa t}\right)
	\qquad(t>0,\ z\in\R^d).
\end{equation}

We will use the Banach norms on random fields as defined
in \cite{FK}. Specifically, we define, for all $k\ge 1$, $\delta>0$, 
and random fields $Z$,
\begin{equation}\label{eq:M}
	\mathcal{M}_\delta^{(k)}(Z) := \sup_{\substack{t\ge 0\\
	x\in\R^d}} \left[ \e^{-\delta t}\|Z_t(x)\|_k\right],
\end{equation}
where we write
\begin{equation}
	\|Z\|_k := \left( \E\left( |Z|^k\right) \right)^{1/k}
	\ \text{whenever $Z\in L^k(\P)$ for some $k\in[1\,,\infty)$}.
\end{equation}

Throughout, $\mathcal{S}$ denotes the collection of all
rapidly-decreasing Schwarz test functions from $\R^d$ to $\R$,
and our Fourier transform is normalized so that
\begin{equation}
	\hat{g} (\xi) = \int_{\R^d} \e^{ix\cdot \xi} g(x)\,\d x
	\qquad\text{for all $g\in L^1(\R^d)$}.
\end{equation}

On several occasions, we apply the Burkholder--Davis--Gundy
inequality 
\cite{Burkholder,BDG,BG}
for continuous $L^2(\P)$ martingales:
If $\{X_t\}_{t\ge 0}$ is a continuous $L^2(\P)$ martingale with
running maximum $X^*_t:=\sup_{s\in[0,t]}|X_s|$
and quadratic variation process $\langle X\rangle$, then for all 
real numbers $k\ge 2$ and $t>0$,
$$
	\left\| X_t^*\right\|_k \le
	\left\| 4k\langle X\rangle_t \right\|_{k/2}^{\nicefrac12}.
	\eqno(\textnormal{BDG})
$$
The factor $4k$ is the asymptotically-optimal bound of
Carlen and Kree \cite{CK} for the sharp constant in the 
Burkholder--Davis--Gundy
inequality that is due to Davis \cite{Davis}. We will also sometimes use the notation
\begin{equation}\label{ubar}
\underline{u}_0:=\inf_{x\in \R^d}u_0(x),\qquad \overline{u}_0:=\sup_{x\in \R^d} u_0(x).
\end{equation}

\section{Main results}
Throughout, we assume tacitly that $\hat f$
is a measurable function [which then is necessarily nonnegative] and
\begin{equation}\label{cond:Dalang}
	\int_{\R^d}  \frac{\hat{f}(\xi)}{1+\|\xi\|^2}\,\d\xi<\infty.
\end{equation}
Condition \eqref{cond:Dalang} ensures the existence of an a.s.-unique 
predictable random field 
$u=\{u_t(x)\}_{t>0,x\in\R^d}$ that solves \textnormal{(SHE)} in the mild 
form \cite{Dalang:99}.\footnote{Dalang's theory assumes that $f$
is continuous away from the origin; this continuity condition
can be removed \cite{FK:TRAMS,PZ}.}
That is, $u$ solves the following random integral equation
for all $t>0$ and $x\in\R^d$:
\begin{equation}\label{mild}
	u_t(x) = (p_t*u_0)(x) + \int_{(0,t)\times\R^d}
	p_{t-s}(y-x)\sigma(u_s(y))\, F(\d s\,\d y)
	\qquad\text{a.s.}
\end{equation}

We note that, because $f$ is positive definite, Condition \eqref{cond:Dalang}
is verified automatically [for all $d\ge 1$]
when $f$ is a bounded function. In fact, it has been shown in
Foondun and Khoshnevisan \cite{FK:TRAMS} that Dalang's 
condition \eqref{cond:Dalang}
is equivalent to the condition that the correlation function $f$ has a 
\emph{bounded potential} in the sense of classical potential theory. 
Let us recall what this means next: Define $R_\beta$ to be the
{\rm $\beta$-potential} corresponding to the convolution
semigroup defined by $\{p_t\}_{t>0}$;  that is, $R_\beta$ is the linear
operator that is defined via setting
\begin{equation}\label{eq:R_beta}
	(R_\beta \phi)(x) := \int_0^\infty\e^{-\beta t}
	(p_t*\phi)(x)\,\d t\qquad(t>0,\,x\in\R^d),
\end{equation}
for all measurable $\phi:\R^d\to\R_+$. Then, Dalang's condition
\eqref{cond:Dalang} is equivalent to the condition that
$R_\beta f$ is a bounded function for one, hence all, $\beta>0$;
and another equivalent statement [the maximum principle] is that
\begin{equation}
	(R_\beta f)(0)<\infty\quad\text{for one, hence all, $\beta>0$}.
\end{equation}
See \cite[Theorem 1.2]{FK:TRAMS} for details.

Our first main result states that if $u_0$ decays at infinity faster than
exponentially, then a mild condition on $f$ ensures that the
solution to \textnormal{(SHE)} is bounded at all times.

\begin{theorem}\label{th:boundedness}
	Suppose $\limsup_{\|x\|\to\infty}
	\|x\|^{-1}\log|u_0(x)|=-\infty$ and 
	$\int_0^1 s^{-a} (p_s*f)(0)\,\d s<\infty$ for some $a\in (0,\nicefrac{1}{2})$. Also assume $\sigma(0)=0$.  Then
	$\sup_{x\in\R^d}|u_t(x)|<\infty$ a.s.\ for all $t>0$.
	In fact, 
	$\sup_{x\in\R^d}|u_t(x)|\in L^k(\P)$ for all $t>0$ and $k\in[2\,,\infty)$.
\end{theorem}

Our condition on $f$ is indeed mild, as the following remark shows.

\begin{remark}
	Suppose that  there exist constants $A\in(0\,,\infty)$ and
	$\alpha\in(0\,,d\wedge 2)$  such that $\sup_{\|x\|>z}f(x) \le A z^{-\alpha}$
	for all $z>0$. 
	[Just about every correlation function that one would like to
	consider has this property.]
	Then we can deduce from the form of the heat kernel that for all $r,s>0$,
	\begin{align}\notag
		(p_s*f)(0)  &\le (2\pi\varkappa s)^{-d/2}
			\cdot\int_{\|x\|\le r}f(x)\,\d x
			+ \sup_{\|x\|>r}f(x)\\
		&\le (2\pi\varkappa s)^{-d/2}
			\cdot\sum_{k=0}^\infty \int_{2^{-k-1}r<\|x\|\le 2^{-k}r}
			f(x)\,\d x + \frac{A}{r^\alpha}\\\notag
		&\le \frac{\textnormal{const}}{s^{d/2}}\cdot \sum_{k=0}^\infty
			\left( 2^{-k-1}r\right)^{d-\alpha} + \frac{A}{r^\alpha}
			\le \textnormal{const}\cdot\left[
			\frac{r^{d-\alpha}}{s^{d/2}}+r^{-\alpha}\right].
	\end{align}
	We optimize over $r>0$ to find that
	$(p_s*f)(0) \le \textnormal{const}\cdot s^{-\alpha/2}.$
	In particular, $(R_\beta f)(0)<\infty$ for all $\beta>0$,
	and $\int_0^1 s^{-a}(p_s*f)(0)\,\d s<\infty$ for some $a\in (0,\nicefrac{1}{2})$.
	\qed
\end{remark}

Recall that the initial function $u_0$ is assumed to be bounded throughout.
For the remainder of our analysis we study only 
bounded initial functions that also satisfy $\inf_{x\in\R^d} u_0(x)>0$.
And we study only correlation functions $f$ that have the form $f= h*\tilde{h}$ for
some nonnegative function $h\in  W^{1,2}_{\textit{loc}}(\R^d)$,
where  $\tilde{h}(x) := h(-x)$
denotes the reflection of $h$, and
$W^{1,2}_{\textit{loc}}(\R^d)$ denotes the vector space
of all locally integrable functions
$g:\R^d\to\R$ whose Fourier transform is a function
that satisfies
\begin{equation}
	\int_{\|x\|<r} \|x\|^2\, |\hat{g}(x)|^2\,\d x<\infty
	\qquad\text{for all $r>0$}.
\end{equation}

Because $L^2(\R^d)\subset W^{1,2}_{\textit{loc}}(\R^d)$,
Young's inequality tells us that $f:= h*\tilde{h}$ 
is positive definite and continuous, provided that 
$h\in L^2(\R^d)$; in that case, 
we have also that $\sup_{x\in\R^d}|f(x)|=f(0)<\infty$. 
And the condition that $h\in L^2(\R^d)$
cannot be relaxed, as there exist many choices of
nonnegative $h\in W^{1,2}_{\textit{loc}}(\R^d)\setminus L^2(\R^d)$
for which  $f(0)=\infty$; see Example \ref{ex:Riesz} below. We
remark also that \eqref{cond:Dalang} holds automatically when $h\in L^2(\R^d)$.

First, let us consider the case that $h\in L^2(\R^d)$ is nonnegative
[so that $f$ is nonnegative,
bounded and continuous, and \eqref{cond:Dalang} is valid automatically]. 
According to the theory of Walsh \cite{Walsh}, \textnormal{(SHE)}
has a mild solution $u=\{u_t(x)\}_{t>0,x\in\R^d}$---for 
all $d\ge 1$---that has continuous trajectories and
is unique up to evanescence among all predictable random
fields that satisfy $\sup_{t\in(0,T)}\sup_{x\in\R^d}\E(|u_t(x)|^2)<\infty$
for all $T>0$. In particular, $u$ solves \eqref{mild}
almost surely for all $t>0$ and $x\in\R^d$, where the stochastic
integral is the one defined by Walsh \cite{Walsh} and Dalang \cite{Dalang:99}.

Our next result describes the behavior of that solution, for nice
choices of $h\in L^2(\R^d)$, when viewed
very far away from the origin.

\begin{theorem}\label{th:sigma:bdd}
	Consider  \textnormal{(SHE)} where $\inf_{x\in\R^d}u_0(x)>0$,
	and suppose $f=h*\tilde{h}$ for a nonnegative
	 $h\in L^2(\R^d)$ that satisfies
	the following for some $a>0$:
	$\int_{\|z\|>n} [h(z)]^2\,\d z=O(n^{-a})$ as $n\to\infty$.
	If $\sigma$ is bounded uniformly away from zero, then
	\begin{equation}\label{eq:sigma:bdd1}
		\limsup_{\|x\|\to\infty}\frac{|u_t(x)|}{(\log\|x\|)^{\nicefrac14}}>0
		\qquad\text{a.s.}\quad\text{for all $t>0$}.
	\end{equation}
	If $\sigma$ is bounded uniformly away
	from zero and infinity, then
	\begin{equation}\label{eq:sigma:bdd2}
		0<\limsup_{\|x\|\to\infty} \frac{|u_t(x)|}{(\log \|x\|)^{%
		\nicefrac12}}<\infty
		\qquad\text{a.s.}\quad\text{for all $t>0$.}
	\end{equation}
\end{theorem}

\begin{remark}
	Our derivation of Theorem \ref{th:sigma:bdd} will in fact yield 
	a little more information. Namely, that the limsups in
	\eqref{eq:sigma:bdd1} and \eqref{eq:sigma:bdd2} are both
	bounded below by a constant
	$c(\varkappa):=c(t\,,\varkappa\,,f\,,d)$ which satisfies 
	$\inf_{\varkappa\in(0,\varkappa_0)}c(\varkappa)>0$
	for all $\varkappa_0>0$; and the limsup in \eqref{eq:sigma:bdd2}
	is bounded above by a constant that does not depend on 
	the viscosity coefficient $\varkappa$.
	\qed
\end{remark}

If $g_1,g_2,\ldots$ is a sequence of independent standard normal
random variables, then it is well known that
$\limsup_{n\to\infty} (2\log n)^{-\nicefrac12} g_n=1$ a.s. Now choose
and fix some $t>0$. Because
$\{u_t(x)\}_{x\in\R^d}$ is a centered Gaussian process
when $\sigma$ is a constant, the preceding theorem
suggests that the asymptotic behavior of $x\mapsto u_t(x)$ is the
same as in the case that $\sigma$ is a constant; and that behavior is
``Gaussian.'' This ``Gaussian'' property continues to hold if we replace
$\dot{F}$ by space-time white noise---that is formally when $f=\delta_0$;
see \cite{CJK}. Next we exhibit ``non Gaussian'' behavior by considering
the following special case of \textnormal{(SHE)}:
$$
	\frac{\partial}{\partial t} u_t(x) = \frac{\varkappa}{2}(\Delta u_t)(x)
	+  u_t(x) \dot{F}_t(x).\eqno(\textnormal{PAM})
$$
This is the socalled ``parabolic Anderson model,''  and arises in
many different contexts in mathematics and theoretical physics
\cite[Introduction]{CM}.

\begin{theorem}\label{th:pa_f(0)Finite}
	Consider \textnormal{(PAM)}
	when $\inf_{x\in\R^d} u_0(x)>0$
	and $f=h*\tilde{h}$ for some nonnegative 
	function $h\in L^2(\R^d)$ that satisfies the following for some $a>0$:
	$\int_{\|z\|>n} [h(z)]^2\,\d z=O(n^{-a})$ as $n\to\infty$,
	Then for every $t>0$ there exist positive and finite constants
	$\underline{A}_t(\varkappa):= \underline{A}(t\,,\varkappa\,,d\,,f\,,a)$
	and $\overline{A}_t= \overline{A}(t\,,d\,,f(0)\,,a)$ such that with probability one
	\begin{equation}
		\underline{A}_t(\varkappa)\le
		\limsup_{\|x\|\to\infty} \frac{\log u_t(x)}{(\log \|x\|)^{\nicefrac12}}
		\le \overline{A}_t.
	\end{equation}
	Moreover: (i) There exists $\varkappa_0:=\varkappa_0(f\,,d)\in(0\,,\infty)$
	such that $\inf_{\varkappa\in(0,\varkappa_0)}
	\underline{A}_t(\varkappa)>0$
	for all $t>0$; and (ii) If $f(x)>0$ for all $x\in\R^d$, then
	$\inf_{\varkappa\in(0,\varkappa_1)}
	\underline{A}_t(\varkappa)>0$ for all $\varkappa_1>0$.
\end{theorem}

The conclusion of Theorem \ref{th:pa_f(0)Finite}
is that, under the condition of that theorem, and if 
the viscosity coefficient $\varkappa$ is
sufficiently small, then for all $t>0$,
\begin{equation}\label{eq:pa_limsup:psi}
	\frac{\underline{B}}{\varkappa^{2\psi-1}}\le
	\limsup_{\|x\|\to\infty} \frac{\log u_t(x)}{(\log \|x\|)^\psi}
	\le \frac{\overline{B}}{\varkappa^{2\psi-1}}
	\qquad\text{a.s.},
\end{equation}
with nontrivial constants $\underline{B}$
and $\overline{B}$ that depend  on $(t\,,d\,,f)$---but \emph{not}
on $\varkappa$---and $\psi=\nicefrac12$. Loosely speaking, the preceding
and its proof together imply  that 
\begin{equation}
	\sup_{\|x\|<R} u_t(x) \mathop{\approx}\limits^{(\log)} \e^{
	\text{const}\cdot (\log R)^{\nicefrac12}},
\end{equation}
for all $\varkappa$ small and $R$ large. This informal assertion was
mentioned earlier in Introduction.

In \cite{CJK} we have proved that if $\dot{F}$ is replaced with
space-time white noise---that is, loosely speaking, when $f=\delta_0$---then
\eqref{eq:pa_limsup:psi} holds with $\psi=\nicefrac23$. That is,
\begin{equation}
	\sup_{\|x\|<R} u_t(x)\mathop{\approx}\limits^{(\log)}  
	\e^{ \text{const} \cdot (\log R)^{\nicefrac23} /\varkappa^{\nicefrac13}},
\end{equation}
for all $\varkappa>0$  and $R$ large.

In some sense these two examples signify the extremes among all choices of
possible correlations. One might wonder if there are other correlation
models that interpolate between the mentioned cases of
$\psi=\nicefrac12$ and $\psi=\nicefrac23$. Our next theorem shows that the answer is
``yes for every $\psi\in(\nicefrac12\,,\nicefrac23)$ when $d=1$
and every $\psi\in(\nicefrac12\,,1)$ when $d\ge 2$.'' However, our construction
requires us to consider certain correlation functions $f$ that have the form
$h*\tilde{h}$ for some $h\in W^{1,2}_{\textit{loc}}(\R^d)\setminus
L^2(\R^d)$. 

In fact, we choose and fix some number $\alpha\in(0\,,d)$, and consider
correlation functions of the Riesz type; namely,
\begin{equation}\label{eq:f=Riesz}
	f(x) :=  \textnormal{const}\cdot \|x\|^{-\alpha}
	\qquad\text{for all $x\in\R^d$}.
\end{equation}
It is not hard to check that $f$ is a correlation function that
has the form $h*\tilde{h}$ for some $h\in W^{1,2}_{\textit{loc}}(\R^d)$,
and $h\not\in L^2(\R^d)$;
see also Example \ref{ex:Riesz} below. 
Because the Fourier transform of $f$ is proportional to
$\|\xi\|^{-(d-\alpha)}$, \eqref{cond:Dalang} is equivalent
to the condition that $0<\alpha<\min(d\,,2)$,
and  Dalang's theory \cite{Dalang:99} tells us that if $u_0:\R^d\to\R$ is bounded and
measurable, then \textnormal{(SHE)} has a solution [that is also unique up
to evanescence], provided that $0<\alpha<\min(d\,,2)$.
Moreover, when $\sigma$ is a constant, \textnormal{(SHE)} has a solution
if and only if $0<\alpha<\min(d\,,2)$. 

Our next result describes the ``non Gaussian'' asymptotic behavior
of the solution to the parabolic Anderson model \textnormal{(PAM)} under
these conditions.

\begin{theorem}\label{th:pa_Riesz}
	Consider \textnormal{(PAM)} when $\inf_{x\in\R^d}u_0(x)>0$. 
	If $f(x)=\textnormal{const}\cdot\|x\|^{-\alpha}$
	for some $\alpha\in(0\,,d\wedge 2)$, then for every $t>0$
	there exist positive and finite constants
	$\underline{B}$
	and $\overline{B}$---both depending
	only on $(t\,,d\,,\alpha)$---such that \eqref{eq:pa_limsup:psi}
	holds with $\psi := 2/(4-\alpha)$; that is, for all $t>0$,
	\begin{equation}
		\frac{\underline{B}}{\varkappa^{\alpha/(4-\alpha)}}\le
		\limsup_{\|x\|\to\infty} \frac{\log u_t(x)}{(\log \|x\|)^{2/(4-\alpha)}}
		\le \frac{\overline{B}}{\varkappa^{\alpha/(4-\alpha)}}
		\qquad\text{a.s.}
	\end{equation}
\end{theorem}
\begin{remark}We mention here that the constants in the above theorems might depend on $u_0$ but \textit{only} through $\inf_{x\in \R^d} u_0(x)$ and $\sup_{x\in \R^d}u_0(x)$. We will not keep track of this dependence. Our primary interest is the dependence on $\varkappa$. \qed
\end{remark}

An important step in our arguments is to show that
if $x_1,\ldots,x_N$ are sufficiently spread out then typically 
$u_t(x_1),\ldots,u_t(x_N)$ are sufficiently close to being independent.
This amounts to a sharp estimate for the socalled ``correlation length.''
We estimate that, roughly using the  arguments of \cite{CJK}, devised 
for the space-time white noise. Those arguments are in turn using
several couplings \cite{Durrett,Liggett}, which might be of some interest.
We add that the presence of spatial correlations adds a number of subtle
[but quite serious] technical problems to this program. 

\section{A coupling of the noise}

\subsection{A construction of the noise}
Let $W:=\{W_t(x)\}_{t\ge 0,x\in\R^d}$ denote $(d+1)$-parameter
Brownian sheet. That is, $W$ is a centered Gaussian random field with
the following covariance structure: For all $s,t\ge 0$
and $x,y\in\R^d$,
\begin{equation}
	\textnormal{Cov}\left( W_t(x)\,, W_s(y) \right) =(s\wedge t)\cdot
	\prod_{j=1}^d(|x_j|\wedge |y_j|)\1_{(0,\infty)}(x_jy_j).
\end{equation}

Define $\mathcal{F}_t$
to be the sigma-algebra generated by all random variables
of the form $W_s(x)$, as $s$ ranges over $[0\,,t]$ and $x$ over
$\R^d$. As is standard in stochastic analysis, we may assume
without loss of generality that $\{\F_t\}_{t\ge 0}$ satisfy the 
``usual conditions''  of the general theory of stochastic processes
\cite[Chapter 4]{DM}.

If $h\in L^2(\R^d)$, then we may
consider the mean-zero Gaussian random field 
$\{(h*W_t)(x)\}_{t\ge 0,x\in\R^d}$ that is defined as the
following Wiener integral:
\begin{equation}
	(h*W_t)(x) := \int_{\R^d} h(x-z)\, W_t(\d z).
\end{equation}
It is easy to see that the covariance function of this process is given by
\begin{equation}
	\textrm{Cov}\left( (h*W_t)(x)\,, (h*W_s)(y)\right)=
	(s\wedge t) f(x-y),
\end{equation}
where we recall, from the introduction, that $f:=h*\tilde{h}$.
In this way we can define an isonormal noise $F^{(h)}$ via
the following: For every $\phi\in\mathcal{S}$ [the usual space of
all test functions of rapid decrease],
\begin{equation}
	F^{(h)}_t(\phi) := \int_{(0,t)\times\R^d}
	\phi(x)(h*\d W_s)(x)\,\d x
	\qquad(t>0).
\end{equation}

It is easy to see that the following form of the stochastic Fubini theorem
holds:
\begin{equation}\label{eq:Fubini}
	F^{(h)}_t (\phi)=\int_{(0,t)\times\R^d} (\phi* \tilde{h})(x)
	\, W(\d s\, \d x).
\end{equation}
[Compute the $L^2(\P)$-norm of the difference.] In particular,
$\{F^{(h)}_t(\phi)\}_{t\ge 0}$ is a Brownian motion [for each
fixed $\phi\in\mathcal{S}$], normalized so that
\begin{equation}\label{eq:Var(Fh)}
	\textnormal{Var}\left( F^{(h)}_1(\phi) \right) = 
	\int_{\R^d} \left| (\phi*\tilde{h})(x)\right|^2\,\d x
	=\frac{1}{(2\pi)^d}\int_{\R^d} |\hat{\phi}(\xi) |^2
	\hat{f}(\xi)\,\d\xi.
\end{equation}
[The second identity is a consequence of Plancherel's theorem,
together with the fact that $|\hat{h}(\xi)|^2=\hat{f}(\xi)$.]
 
 \subsection{An extension}
 Suppose $h\in L^2(\R^d)$, and that the underlying correlation function
 is described by $f:= h*\tilde{h}$.
 Consider the following probability density function on $\R^d$:
\begin{equation}
	\varrho (x) := \prod_{j=1}^d \left(\frac{1-\cos x_j}{\pi x_j^2}\right)
	\qquad\text{for $x\in\R^d$}.
\end{equation}
We may build an approximation $\{\varrho_n\}_{n\ge 1}$ 
to the identity as follows:
For all real numbers $n\ge 1$ and for every $x\in\R^d$,
\begin{equation}
	\varrho_n(x) := n^d\varrho(nx),
	\quad\text{so that}\quad\hat{\varrho}_n(\xi) = \prod_{j=1}^d \left(
	1-\frac{|\xi_j|}{n}\right)^+,
\end{equation}
for all $\xi \in \R^d$.

\begin{lemma}\label{lem:Fh1-Fh2}
	If $h\in L^2(\R^d)$, then for all $\phi\in\mathcal{S}$
	and integers $n,m\ge 1$,
	\begin{equation}\begin{split}
		&\E\left( \sup_{t\in(0,T)}
			\left| F^{(h*\varrho_{n+m})}_t(\phi) 
			- F^{(h*\varrho_n)}_t(\phi) \right|^2\right)\\
		&\hskip1.8in \le \frac{16d^2 T}{(2\pi)^d}\int_{\R^d} |\hat{\phi}(\xi)|^2
			\left( 1 \wedge \frac{\|\xi\|^2}{n^2}\right) \hat{f}(\xi)\,\d\xi.
	\end{split}\end{equation}
\end{lemma}

\begin{proof}
	By  the Wiener isometry and Doob's maximal inequality,
	the left-hand side of the preceding display is bounded above by $4TQ$,
	where
	\begin{equation}\begin{split}
		Q &:= \int_{\R^d}
			\left| \left( \phi*\widetilde{h*\varrho}_{n+m} \right)
			(x) -  \left(\phi*\widetilde{h*\varrho}_n\right)(x) \right|^2\,\d x\\
		&=\frac{1}{(2\pi)^d} \int_{\R^d} |\hat\phi(\xi)|^2 \left|
			\hat{\varrho}_{n+m}(\xi) - \hat{\varrho}_n(\xi)\right|^2\hat{f}(\xi)\,
			\d\xi;
	\end{split}\end{equation}
	we have appealed to the Plancherel's theorem, together with the fact that
	$\hat f (\xi)= |\hat h (\xi)|^2$.
	Because
	\begin{equation} \label{eqrho}
		0\le 1-\hat{\varrho}_n(\xi) \le 1- 
		\left(\left(1-\frac{1}{n}\max_{1\le j\le d}|\xi_j|\right)^+\right)^d
		\le \frac{d\|\xi\|}{n},
	\end{equation}
	it follows from the triangle inequality that
	$| \hat{\varrho}_{n+m}(\xi)- \hat{\varrho}_n(\xi) |
	\le  2d\|\xi\|/n$.
	This implies the lemma, because we also have $| \hat{\varrho}_{n+m}(\xi)-
	\hat{\varrho}_n(\xi) |\le \| \varrho_{n+m}\|_{L^1(\R^d)}+
	\|\varrho_n\|_{L^1(\R^d)}=2\le 2d$.
\end{proof}

Lemma \ref{lem:Fh1-Fh2} has the following consequence:
Suppose $h\in W^{1,2}_{\textit{loc}}(\R^d)$, and $f:= h*\tilde{h}$
in the sense of generalized functions. Because $h\in W^{1,2}_{\textit{loc}}(\R^d)$,
the dominated convergence theorem tells us that
\begin{equation}
	\lim_{n\to\infty }\int_{\R^d}|\hat\phi(\xi)|^2\left(
	1\wedge \frac{\|\xi\|^2}{n^2}\right)\hat{f}(\xi)\,\d\xi=0
	\quad\text{for all $\phi\in\mathcal{S}$}.
\end{equation}
Consequently,
$F_t^{(h)}(\phi):=\lim_{n\to\infty} F_t^{(h*\varrho_n)}(\phi)$
exists in $L^2(\P)$, locally uniformly in $t$. Because $L^2(\P)$-limits
of centered Gaussian random fields are themselves Gaussian, it
follows that $F^{(h)}:=\{F^{(h)}_t(\phi)\}_{t\ge 0,\phi\in\mathcal{S}}$
is a centered Gaussian random field, and $\{F^{(h)}_t\}_{t\ge 0}$ is
a Brownian motion scaled in order to satisfy \eqref{eq:Var(Fh)}.
We mention also that, for these very reasons, $F^{(h)}$ satisfies
\eqref{eq:Fubini}  a.s.\  for all $t\ge 0$ and $\phi\in\mathcal{S}$.
The following example shows that one can construct the Gaussian
random field $F^{(h)}$ even when
$h\in W^{1,2}_{\textit{loc}}(\R^d)$ is not in $L^2(\R^d)$.

\begin{example}[Riesz kernels]\label{ex:Riesz}
	We are interested in correlation functions of the Riesz type:
	$f(x)=c_0\cdot\|x\|^{-\alpha}$, where
	$x\in\R^d$ [and of course $\alpha\in(0\,,d)$ so that 
	$f$ is locally integrable]. If is well known that $\hat{f}(\xi)=
	c_1\cdot\|\xi\|^{-(d-\alpha)}$ for a positive and finite
	constant $c_1$ that depends only on $(d\,,\alpha\,,c_0)$. We may define 
	$h\in L^1_{\textit{loc}}(\R^d)$ via $\hat h(\xi):= c_1^{\nicefrac12}\cdot
	\|\xi\|^{-(d-\alpha)/2}.$ It then follows that $f=h*\tilde{h}$;
	and it is clear from the fact that $\hat f = |\hat h|^2$ that
	$h\in W^{1,2}_{\textit{loc}}(\R^d)$ if and only if
	$\int_{\|\xi\|<1} \|\xi\|^2 \hat{f}(\xi)\,\d\xi<\infty,$
	which is  satisfied automatically because $\alpha\in(0\,,d)$.
	\qed
\end{example}

Of course, even more general Gaussian random fields can be
constructed using only general theory. What is important for the
sequel is that here we have constructed a random-field-valued
\emph{stochastic process}
$(t\,,h)\mapsto F^{(h)}_t$; i.e., the random fields 
$\{ F_t^{(h)}(\phi)\}_{\phi\in\mathcal{S}}$ are all coupled
together as $(t\,,h)$ ranges over the index set
$(0\,,\infty)\times W^{1,2}_{\textit{loc}}(\R^d)$.

\subsection{A coupling of stochastic convolutions}

Suppose $Z:=\{Z_t(x)\}_{t\ge 0,x\in\R^d}$ is a random field
that is predictable with respect to the filtration $\F$, and satisfies
the following for all $t>0$ and $x\in\R^d$:
\begin{equation}\label{cond:PRF}
	\int_0^t\d s\mathop{\iint}\limits_{\R^d\times\R^d} \d y\,\d z\
	p_{t-s}(y-x)p_{t-s}(z-x) \left| \E\left( Z_s(y)Z_s(z)\right) \right|
	f(y-z)<\infty.
\end{equation}
Then we may apply the  theories of Walsh \cite[Chapter 2]{Walsh} and
Dalang \cite{Dalang:99} to the martingale
measure $(t\,,A)\mapsto F^{(h)}_t(\1_A)$, and construct the
stochastic convolution $p*Z\dot{F}^{(h)}$ as the
random field
\begin{equation}
	\left( p*Z\dot{F}^{(h)}\right)_t(x) := 
	\int_{(0,t)\times\R^d} p_{t-s}(y-x)Z_s(y)\, F^{(h)}(\d s\,\d y).
\end{equation}
Also, we have the following It\^o-type isometry:
\begin{align}
	&\E\left(\left|
		\int_{(0,t)\times\R^d} p_{t-s}(y-x)Z_s(y)\, F^{(h)}(\d s\,\d y)
		\right|^2\right)\label{eq:isometry}\\\notag
	&= \int_0^t\d s\int_{\R^d}\d y\int_{\R^d}\d z\
		p_{t-s}(y-x)p_{t-s}(z-x) \E\left[ Z_s(y)Z_s(z)\right]
		f(y-z).
\end{align}

If $h:\R^d\to\R_+$ is nonnegative and measurable, then we define, 
for all real numbers $n\ge 1$,
\begin{equation} \label{rho}
	h_n(x) := h(x) \hat{\varrho}_n(x)
	\qquad\text{for every $x\in\R^d$}.
\end{equation}
Some important features of this construction are that: 
(a) $0\le h_n\le h$ pointwise;  (b)
$h_n\to h$ as $n\to\infty$, pointwise;  
(c) every $h_n$ has compact support; and (d) if $h\in W^{1,2}_{\textit{loc}}(\R^d)$,
then $h_n\in W^{1,2}_{\textit{loc}}(\R^d)$ for all $n\ge 1$.

For the final results of this section we consider only nonnegative 
functions $h\in L^2(\R^d)$ that satisfy the following [relatively mild]
condition:
\begin{equation}\label{eq:h:a}
	\sup_{r>0}\left[
	r^a \cdot \int_{\|x\|>r} [h(x)]^2\,\d x \right] <\infty
	\qquad\text{for some $a>0$}. 
\end{equation}

\begin{lemma}\label{lem:Omega}
	If $h\in L^2(\R^d)$ satisfies \eqref{eq:h:a},
	then there exists $b\in(0\,,2)$ such that
	\begin{equation}
		\sup_{n\ge 1}\left[ 
		n^b \cdot \int_{\R^d} \left( 1\wedge \frac{\|x\|^2}{n^2}\right)
		\, [h(x)]^2\,\d x \right]<\infty.
	\end{equation}
\end{lemma}

\begin{proof}
	We may---and will---assume, without loss of generality,  that 
	\eqref{eq:h:a} holds for some $a\in(0\,,2)$. Then, thanks to \eqref{eq:h:a},
	\begin{equation}\begin{split}
		\int_{\|x\|\le n} \frac{\|x\|^2}{n^2} 
			[h(x)]^2\,\d x &\le 
			\sum_{k=0}^\infty 4^{-k}
			\mathop{\int}_{2^{-k-1}n <\|x\|\le 2^{-k}n}
			[h(x)]^2\,\d x\\
		&\le \textnormal{const}\cdot\sum_{k=0}^\infty 4^{-k}
			\left( 2^{-k-1}n\right)^{-a},
	\end{split}\end{equation}
	and this is $O(n^{-a})$ since $a\in(0\,,2)$. The lemma follows
	readily from this.
\end{proof}

\begin{proposition}\label{pr:mk}
	If $h\in L^2(\R^d)$  is nonnegative  and
	satisfies \eqref{eq:h:a}, then
	for all predictable random fields that satisfy
	\eqref{cond:PRF}, and for all $\delta>1$, $x\in\R^d$, $n\ge 1$, and $k\ge 2$,
	\begin{equation}\label{eq:mk}
		\mathcal{M}_\delta^{(k)}\left(p*Z
		\dot{F}^{(h)}-p*Z\dot{F}^{(h_n)}\right)
		\le C \sqrt{\frac{k}{n^b}}\ \mathcal{M}_\delta^{(k)}(Z)
	\end{equation}
	for some positive constant $C$ which does not depend on $\varkappa$,
	where $b$ is the constant introduced in Lemma \ref{lem:Omega} and $\mathcal{M}_{\delta}^{(k)}$ is defined in \eqref{eq:M}.
\end{proposition}

\begin{remark}
	This proposition has a similar appearance as Lemma \ref{lem:Fh1-Fh2}.
	However, note that here we are concerned with correlations functions
	of the form $q*\tilde{q}$ where $q:=h\hat{\varrho}_n$, whereas
	in Lemma \ref{lem:Fh1-Fh2} we were interested in $q=h*\varrho_n$.
	The methods of proof are quite different. \qed
\end{remark}
	
\begin{proof}
	The present proof follows closely renewal-theoretic ideas 
	that were developed in \cite{FK}. Because we wish to appeal to the same
	method several more times in the sequel, we describe nearly 
	all the details  once, and then refer to the present discussion for
	details in later applications of this method.
	
	Eq.\ \eqref{eq:Fubini} implies that  $p* Z \dot{F}^{(h)} - p*Z \dot{F}^{(h_n)}
	= p* Z \dot{F}^{(D)}$ a.s.,
	where $D:=h-h_n=h(1-\hat{\varrho}_n)\ge 0$.
	According to \textnormal{(BDG)},
	\begin{align}
		&\E\left( \left| \int_{(0,t)\times\R^d}
			p_{t-s}(y-x) Z_s(y) F^{(D)}(\d s\,\d y) \right|^k\right)\\
			\notag
		&\le \E\left(\left| 4k
			\int_0^t\d s\mathop{\iint}\limits_{\R^d\times\R^d}
			\d y\,\d z\ p_{t-s}(y-x)p_{t-s}(z-x)
			\mathcal{Z} f^{(D)}(y-z)\right|^{k/2} \right),
	\end{align}
	where $\mathcal{Z} := |Z_s(y)Z_s(z)|$ and
	$f^{(D)}:=D*\tilde{D}$; we observe that $f^{(D)}\ge 0$.
	The classical Minkowski inequality for integrals implies that 
	$\| \int_{(0,t)\times\R^d\times\R^d}(\,\cdots)\|_{k/2}\le
	\int_{(0,t)\times\R^d\times\R^d}\|\cdots\|_{k/2}$.
	Therefore, it follows that
	\begin{align}
		&\label{4.24}\E\left( \left| \int_{(0,t)\times\R^d}
			p_{t-s}(y-x) Z_s(y) F^{(D)}(\d s\,\d y) \right|^k\right)\\
			\notag
		&\le\left|4k 
			\int_0^t\d s\!\!\!\!\mathop{\iint}\limits_{\R^d\times\R^d}
			\d y\,\d z\ p_{t-s}(y-x)p_{t-s}(z-x) f^{(D)}(z-y)
			\| Z_s(y)Z_s(z)\| _{k/2} \right|^{k/2}.		
	\end{align}
	Young's inequality shows that the function $f^{(D)}=D*\tilde{D}$ is  
	bounded uniformly from above by
	\begin{align}
		\|D\|_{L^{2}(\R^d)}^2 &= 
			\|h(1-\hat{\varrho}_n)\|^2_{L^2(\R^d)} \\ \notag
		&\le \left(\frac{d}{n}\right)^2\int_{|z|_{\infty} \le n } 
			[\|z\|h(z)]^2dz +\int_{|z|_{\infty}>n}[h(z)]^2 dz=O(n^{-b}),
	\end{align}
	where $|z|_{\infty}:=\max_{1\le j\le n} |z_j|$; see also
	Lemma \ref{lem:Omega}. Therefore
	\begin{align}	    \label{eq3.24}
		&\E\left( \left| \int_{(0,t)\times\R^d}
			p_{t-s}(y-x) Z_s(y) F^{(D)}(\d s\,\d y) \right|^k\right)\\
			\notag
		&=O(n^{-bk/2}) \left| k 
			\int_0^t\d s\!\!\!\!\mathop{\iint}\limits_{\R^d\times\R^d}
			\d y\,\d z\ p_{t-s}(y-x)p_{t-s}(z-x)
			\| Z_s(y)Z_s(z)\| _{k/2} \right|^{k/2}.
	\end{align}
	According to the Cauchy--Schwarz inequality, $\| Z_s(y) Z_s(z)\|_{k/2}^{\nicefrac12}$
	is bounded above by 
	$\sup_{w\in\R^d}
	\|Z_s(w)\|_k\le\e^{\delta s}\mathcal{M}_\delta^{(k)}(Z),$
	and the proposition follows.
\end{proof}

\section{Moment and tail estimates}
In this section we state and prove a number of inequalities that will be needed
subsequently. Our estimates are developed in different subsections for the different
cases of interest [e.g., $\sigma$ bounded, $\sigma(u)\propto u$, $f=h*\tilde{h}$
for $h\in L^2(\R^d)$, $f(x)\propto \|x\|^{-\alpha}$, etc.]. Although the
techniques vary from one subsection to the next, the common theme of this
section is that all bounds are ultimately derived by establishing moment
inequalities of one sort or another.

\subsection{An upper bound in the general $h\in L^2(\R^d)$ case}

\begin{proposition}\label{pr:ub:general}
	Let $u$ denote the solution to \textnormal{(SHE)}, where
	$f:=h*\tilde{h}$ for some nonnegative $h\in L^2(\R^d)$.
	Then, for all $t>0$ there exists a positive and finite
	constant $\gamma=\gamma(d\,,f(0)\,,t)$---independent
	of $\varkappa$---such that for all $\lambda>\e$,
	\begin{equation}
		\sup_{x\in\R^d}\P\left\{ u_t(x) >\lambda
		\right\} \le \gamma^{-1} \e^{-\gamma(\log\lambda)^2}.
	\end{equation}
\end{proposition}

\begin{proof}
	Because
	$|(p_t*u_0)(x)|\le\|u_0\|_{L^\infty(\R^d)}$ uniformly in $x\in\R^d$,
	we can appeal to \textnormal{(BDG)} and \eqref{mild} in order to obtain
	\begin{align}\begin{split}
		\|u_t(x)\|_k &\le \|u_0\|_{L^\infty(\R^d)}+ 
			\left\| \int_{(0,t)\times\R^d}p_{t-s}(y-s)\sigma(u_s(y))
			F^{(h)}(\d s\,\d y)\right\|_k \\
		& \le \|u_0\|_{L^\infty(\R^d)}+2\sqrt{k}\left(
			\E\left[\left( \int_0^t \d s\mathop{\iint}\limits_{\R^d\times\R^d}\d y
			\,\d z\ \mathcal{Q}\right)^{k/2}\right] \right)^{1/k},
	\end{split}\end{align}
	where
	$\mathcal{Q} := f(y-z) p_{t-s}(y-x)p_{t-s}(z-x)
	\sigma (u_s(y))\sigma (u_s(z) )$;
	see the proof of Proposition \ref{pr:mk} for more details on this method.
	Since $|\mathcal{Q}|$ is bounded above by 
	$\mathcal{W} := f(0) p_{t-s}(y-x)p_{t-s}(z-x)
	|\sigma (u_s(y) )\cdot \sigma (u_s(z) ) |$ we find that
	\begin{equation}
		\| u_t(x)\|_k  \le \|u_0\|_{L^\infty(\R^d)}+ \left(4k
		 \int_0^t \d s\mathop{\iint}\limits_{\R^d\times\R^d}\d y
		\,\d z\ \|\mathcal{W}\|_{k/2} \right)^{\nicefrac12},
	\end{equation}
	Because $|\sigma(z)|\le|\sigma(0)|+\lip|z|$
	for all $z\in\R$, we may apply the Cauchy--Schwarz inequality to find that
	$\| u_t(x)\|_k$ is bounded above by 
	\begin{align}
		&\|u_0\|_{L^\infty(\R^d)}+ \left(4k\cdot f(0)
			\int_0^t \d s\int_{\R^d}\d y \
			p_{t-s}(y-x)  \| \sigma(u_s(y))\|_k^2 \right)^{\nicefrac12}\\\notag
		&   \le \|u_0\|_{L^\infty(\R^d)}+ \left(4k \cdot f(0)
			\int_0^t \d s\int_{\R^d}\d y \
			p_{t-s}(y-x)  \left[ 
			|\sigma(0)|+\lip\|u_s(y)\|_k \right]^2 \right)^{\nicefrac12}.
	\end{align}
	We introduce a parameter $\delta>0$ whose value will be chosen later on.
	It follows from the preceding and some algebra that
	\begin{equation}
		\| u_t(x)\|_k ^2 \le 2\|u_0\|_{L^\infty(\R^d)}^2
		+16kf(0)\left(|\sigma(0)|^2t +\lip^2
		\e^{2\delta t}\mathcal{A}\right),
	\end{equation}
	where $\mathcal{A} := \int_0^t \d s \, \e^{-2\delta (t-s)}
	\int_{\R^d}\d y\ p_{t-s}(y-x)\e^{-2\delta s}\|u_s(y)\|_k^2.$
	Note that 
	\begin{equation}
		\mathcal{A} \le \int_0^t \d s \, \e^{-2\delta(t-s)}
		\int_{\R^d} \d y\ p_{t-s}(y-x) \left[\mathcal{M}_\delta^{(k)}(u)\right]^2
		\le \frac{1}{2\delta}\left[\mathcal{M}_\delta^{(k)}(u)\right]^2.
	\end{equation}
	Therefore, for all $\delta>0$ and $k\ge 2$, $[\mathcal{M}_\delta^{(k)}
	(u)]^2$ is bounded above by
	\begin{equation}
		2\|u_0\|_{L^\infty(\R^d)}^2+ 16kf(0)
		\left( |\sigma(0)|^2\sup_{t\ge 0}
		\left[ t\e^{-2\delta t}\right] +\frac{\lip^2}{2\delta}
		\left[\mathcal{M}_\delta^{(k)}(u)\right]^2\right).
	\end{equation}
	Let us choose $\delta:=\left(1\vee 16 f(0)\lip^2\right)k$ to find that
	$\mathcal{M}_\delta^{(k)}(u)^2 \le (4\sup_{x\in \R^d}u_0(x)^2+C k)$
	for some constant $C>0$ that does not depend on $k$, and hence,
	\begin{equation}\label{mom:upbd}
		\sup_{x\in \R^d}\|u_t(x)\|_k\le 
		\textnormal{const}\cdot\sqrt{k}\,\e^{\left(1\vee 16f(0)\lip^2\right)kt}.
	\end{equation}
	Lemma 3.4 of \cite{CJK} then tells us that there exists 
	$\gamma:=\gamma(t)>0$ sufficiently small [how small depends 
	on $t$  but not on $(\varkappa\,,x)$] such that
	$\E [ \exp ( \gamma  (\log_+ u_t(x)  )^2 )  ]<\infty$.
	Therefore, the proposition follows from Chebyshev's inequality.
\end{proof}

\subsection{Lower bounds for $h\in L^2(\R^d)$ when $\sigma$ is bounded }

\begin{lemma}\label{tail_bound} 
	Let $u$ denote the solution to \textnormal{(SHE)}, where $\sigma$ is assumed
	to be bounded uniformly away from zero and infinity
	and $\inf_{x\in\R^d}u_0(x)>0$. If $f=h*\tilde{h}$
	for some nonnegative $h\in L^2(\R^d)$, then for all $t>0$
	there exist positive and finite constants $c_1=c_1(\varkappa\,,t\,,d\,,f)$
	and $c_2=c_2(t\,,d\,,f)$---independent of $\varkappa$---such that
	uniformly for all $\lambda>\e$,
	\begin{equation}
		c_1^{-1} \e^{-c_1\lambda^2} \le
		\inf_{x\in\R^d}\P\left\{ |u_t(x)|>\lambda\right\}\le
		\sup_{x\in\R^d}\P\left\{ |u_t(x)|>\lambda\right\}
		\le c_2^{-1} \e^{-c_2 \lambda^2}.
	\end{equation}
	Furthermore,
	$\sup_{\varkappa\in (0,\varkappa_0)} c_1(\varkappa)<\infty$ 
	for all $\varkappa_0<\infty$.
 \end{lemma}
 
\begin{proof}
	Choose and fix an arbitrary $\tau>0$, and consider the continuous
	$L^2(\P)$ martingale $\{M_t\}_{t\in[0,\tau]}$ defined by
	\begin{equation}
		M_t := (p_\tau*u_0)(x)+\int_{(0,t)\times\R^d} p_{\tau-s}(y-x)
		\sigma(u_s(y))\, F^{(h)}(\d s\,\d y),
	\end{equation}
	as $t$ ranges within $(0\,,\tau)$.
	By It\^o's formula, for all even integers $k\ge 2$,
	\begin{equation}\label{eq:M:Ito}
		M_t^k = (p_\tau*u_0)(x)^k+k\int_0^t M_s^{k-1}\,\d M_s+
		\binom{k}{2}\int_0^t M_s^{k-2}\,\d\langle M\rangle_s.
	\end{equation}
	The final integral that involves quadratic variation can be written as
	\begin{equation}
		\int_0^t M_s^{k-2}\left[\int_{\R^d}\d y\int_{\R^d}
		\d z\ p_{\tau -s}(y-x)p_{\tau-s}(z-x) f(z-y)\mathcal{Z}\right]\,\d s,
	\end{equation}
	where $\mathcal{Z} := \sigma(u_s(y))\sigma(u_s(z))\ge \epsilon_0^2$ 
	for some $\epsilon_0>0$. This is because $\sigma$ is uniformly 
	bounded away from $0$.
	Thus, the last integral in \eqref{eq:M:Ito} is bounded below by
	\begin{equation}\begin{split}
		&\epsilon_0^2\int_0^tM_s^{k-2}\left[\int_{\R^d}\d y\int_{\R^d}
			\d z\ p_{\tau -s}(y-x)p_{\tau-s}(z-x) f(z-y)\right]\,\d s\\
		&\hskip1.5in
			=\epsilon_0^2\int_0^tM_s^{k-2}\left( p_{\tau-s}\,, p_{\tau-s}*f
			\right)_{L^2(\R^d)}\,\d s,
	\end{split}\end{equation}
	where $\langle a\,,b \rangle_{L^2(\R^d)}:= \int_{\R^d}a(x)b(x)\,\d x$ 
	denotes the usual inner product on $L^2(\R^d)$. This leads us to
	the recursive inequality,
	\begin{equation}
		\E(M_t^k) \ge \left(\inf_{x\in\R^d}u_0(x)\right)^k
		+\binom{k}{2}\epsilon_0^2\cdot
		\int_0^t \E(M_s^{k-2}) \langle p_{\tau-s}\,, p_{\tau-s}*f
		\rangle_{L^2(\R^d)}\,\d s.
	\end{equation}
	
	Next, consider   the Gaussian process $\{\zeta_t\}_{t\ge 0}$ defined by
	\begin{equation}
		\zeta_t := \epsilon_0\int_{(0,t)\times\R^d} p_{\tau-s}(y-x)
		\, F^{(h)}(\d s\,\d y)
		\quad(0<t<\tau).
	\end{equation}
	We may iterate, as was done in \cite[proof of Proposition 3.6]{CJK}, 
	in order to find that
	\begin{equation}
		\E(M_t^k) \ge 
		\E\left(\left[\inf_{x\in\R^d}u_0(x)+\zeta_t\right]^k \right) \ge 
		\E\left(\zeta_t^k\right)\ge \left(
		\textnormal{const} \cdot k\, \E\left[\zeta_t^2
		\right]\right)^{k/2}.
	\end{equation}
	Now $\E(\zeta_t^2) = \epsilon_0^2\int_0^t \langle p_{\tau-s}\,,
	p_{\tau-s}*f\rangle_{L^2(\R^d)}\d s$.
	Since $p_{\tau-s}\in\mathcal{S}$ for all $s\in(0\,,\tau)$,
	Parseval's identity applies, and it follows that
	\begin{equation}
		\langle p_{\tau-s}\,,p_{\tau-s}*f\rangle_{L^2(\R^d)}
		=\frac{1}{(2\pi)^d}\int_{\R^d}\hat f(\xi)
		\e^{-\varkappa(\tau -s)\|\xi\|^2}\,\d\xi.
	\end{equation}
	Therefore,
	\begin{equation}\begin{split}
		\E(\zeta_\tau^2) &=\frac{\epsilon_0^2}{(2\pi)^d}
			\int_{\R^d}\hat f(\xi)\left[ \frac{
			1-\e^{-\varkappa\tau\|\xi\|^2}}{\varkappa
			\|\xi\|^2}\right]\,\d\xi\\
		&\ge  \frac{\epsilon_0^2}{2(2\pi)^d}\int_{\R^d}
			\frac{\hat f(\xi)}{\tau^{-1}+\varkappa\|\xi\|^2}\,\d\xi.
	\end{split}\end{equation}
	This requires only the elementary bound
	$(1-\e^{-z})/z \ge(2(1+z))^{-1}$, valid for all $z>0$.
	Since $M_t=u_t(x)$ when $t=\tau$, it follows that
	\begin{equation}\label{eq:moment:lb}
		c(\varkappa)\sqrt k \le \inf_{x\in\R^d}\|u_t(x)\|_k,
	\end{equation}
	for all $k\ge 2$, where $c(\varkappa)=
	c(t\,,\varkappa\,,f\,,d)$ is positive and finite,
	and has the additional property that
	\begin{equation}\label{cond:c:kappa}
		\inf_{\varkappa\in(0,\varkappa_0)} c(\varkappa) >0\qquad
		\text{for all $\varkappa_0>0$}.
	\end{equation}
	Similar arguments reveal that 
	\begin{equation}\label{eq:moment:ub}
		\sup_{x\in\R^d}\|u_t(x)\|_k\le c'\sqrt{k},
	\end{equation}
	for all $k\ge 2$, where $c'$ is a positive
	and finite constant that depends only on $(t\,,f\,,d)$. 
	The result follows from
	the preceding two moment estimates (see \cite{CJK} for details).
\end{proof}

\begin{lemma}\label{tail_bound0}
	Let $u$ denote the solution to \textnormal{(SHE)}, where $\sigma$ is assumed
	to be bounded uniformly away from zero
	and $\inf_{x\in\R^d}u_0(x)>0$. If $f=h*\tilde{h}$
	for some nonnegative $h\in L^2(\R^d)$, then for all $t>0$
	there exists a positive and finite constant $a(\varkappa):=
	a(\varkappa\,,t\,,d\,,f)$ such that uniformly for every
	$\lambda>\e$,
	\begin{equation} \label{prop1}
		\P\{|u_t(x)|\ge \lambda\} \ge 
		\frac{\exp\left(-a(\varkappa)\lambda^4 \right)}{\sqrt{a(\varkappa)}}.
	\end{equation}
	Furthermore, $\sup_{\varkappa\in(0,\varkappa_0)}a(\varkappa)<\infty$ for all $\varkappa_0>0$.
\end{lemma}

\begin{proof}
	The proof of this proposition is similar to the proof of 
	Proposition 3.7 in the companion paper \cite{CJK},
	and uses the following elementary fact [called the
	``Paley--Zygmund inequality'']: If $Z\in L^2(\P)$
	is nonnegative and $\epsilon\in(0\,,1)$, then 
	\begin{equation}\label{eq:PZ}
		\P\left\{ Z >(1-\epsilon) \E Z\right\} \ge \frac{(\epsilon \E Z)^2}{
		\E(Z^2)}.
	\end{equation}
	This is a ready consequence of the Cauchy--Schwarz inequality.
	
	Note, first, that the moment bound \eqref{eq:moment:lb}
	continues to hold for a constant $c(\varkappa)=c(t\,,\varkappa\,,f\,,d)$
	that satisfies \eqref{cond:c:kappa}. We can no longer apply
	\eqref{eq:moment:ub}, however, since that inequality used
	the condition that $\sigma$ is bounded above; a property that
	need not hold in the present setting. Fortunately, the
	general estimate \eqref{mom:upbd} is valid with ``const''
	not depending on $\varkappa$. Therefore,
	we appeal to the Paley--Zygmund inequality \eqref{eq:PZ} to see that
	\begin{equation}
		\P\left\{ |u_t(x)|\ge \frac{1}{2}\|u_t(x)\|_{2k} \right\} 
		\ge \frac{\left[\E\left(|u_t(x)|^{2k}\right)
		\right]^2}{4\E\left(|u_t(x)|^{4k}\right)}
		\ge \textnormal{const}\cdot [c(\varkappa)]^2 \e^{-Ck^2},
	\end{equation}
	as $k \to \infty$, where $C\in(0\,,\infty)$ does not depend 
	on $(k\,,\varkappa)$. 
	Since $\|u_t(x)\|_{2k}\ge c(\varkappa)\cdot\sqrt{2k}$, it follows that
	$\P \{ |u_t(x)|\ge c(\varkappa)\cdot\sqrt{k/2} \} 
	\ge \exp  (-C'k^2)$
	as $k \to \infty$ for some $C'$ which depends only on $t$. We obtain the proposition by considering $\lambda$ between
	$c(\varkappa)\cdot\sqrt{k/2}$ and $c(\varkappa)\cdot\sqrt{(k+1)/2}$.
\end{proof}

\subsection{A lower bound for the parabolic Anderson model for
	$h\in L^2(\R^d)$}

Throughout this subsection we consider $u$ to be the solution to
the parabolic Anderson model \textnormal{(PAM)} in the case that 
$\inf_{x\in\R^d}u_0(x)>0$.

\begin{proposition}\label{pr:pa_mombd}
	There exists a constant $\Lambda_d\in(0\,,\infty)$---depending
	only on $d$---such that for all $t,\varkappa>0$ and $k\ge 2$,
	\begin{equation}
		\Big[\inf_{x\in \R^d}u_0(x)\Big]^k\e^{\Lambda_d a_t k^2}\le \E\left( |u_t(x)|^k\right) \le \Big[\sup_{x\in \R^d}u_0(x)\Big]^k \e^{tf(0)k^2},
	\end{equation}
	where $a_t = a_t(f\,,\varkappa)>0$ for all $t,\varkappa>0$, and is defined by
	\begin{equation}\label{eq:A_t:lb}
		a_t := \sup_{\delta>0}\left[ \frac{\delta^2}{4\varkappa}
		\left( 1\wedge \frac{4\varkappa t}{\delta^2}\right)
		\inf_{x\in B(0,\delta)}f(x)
		\right].
	\end{equation}
\end{proposition}

This proves, in particular, that the  exponent  estimate
$\left(1\vee 16f(0)\text{Lip}_{\sigma}^2\right)k^2t$, derived more
generally in \eqref{mom:upbd}, is sharp---up to a constant---as a function of $k$.

The proof of Proposition \ref{pr:pa_mombd} hinges on the following,
which by itself is a ready consequence of a
moment formula of Conus \cite{Conus}; see also \cite{BC,HuNualart} 
for related results and special cases.

\begin{lemma}[\cite{Conus}]\label{lem:DC}
	For all $t>0$,  and $x\in\R^d$, we have the following inequalities
	\begin{equation}\begin{split}
		& \E\left( | u_t(x)|^k\right)\ge\left[\inf_{x\in\R^d}u_0(x)\right]^k
			\cdot  \E\exp\left(
			\mathop{\sum\sum}\limits_{1\le i\ne j\le k}\int_0^t f
			\left( \sqrt\varkappa\left[ b^{(i)}_r - b^{(j)}_r \right]\right)\,\d r
			\right) ,\\
		&\E\left( | u_t(x)|^k\right)\le \left[\sup_{x\in\R^d}u_0(x)\right]^k
			\cdot  \E\exp\left(
			\mathop{\sum\sum}\limits_{1\le i\ne j\le k}\int_0^t f
			\left( \sqrt\varkappa\left[ b^{(i)}_r - b^{(j)}_r \right]\right)\,\d r
			\right),
	\end{split}\end{equation}
	where $b^{(1)},b^{(2)},\ldots$ denote independent 
	standard Brownian motions in $\R^d$. 
\end{lemma}

\begin{proof}[Proof of Proposition \ref{pr:pa_mombd}]
	The upper bound for $\E(|u_t(x)|^k)$ follows readily
	from Lemma \ref{lem:DC} 
	and the basic fact that $f$ is maximized at the origin.
	
	In order to establish the lower bound recall that $f$ is continuous
	and $f(0)>0$.
	Because $f(x)\ge q\1_{B(0,\delta)}(x)$
	for all $\delta>0$, with $q=q(\delta):=\inf_{x\in B(0,\delta)}f(x)$,
	it follows that if $b^{(1)},\ldots,b^{(k)}$ are independent
	$d$-dimensional Brownian motions, then
	\begin{equation}\begin{split}
		&\mathop{\sum\sum}\limits_{1\le i\ne j\le k}\int_0^t f
			\left(\sqrt\varkappa\left[b^{(i)}_r - b^{(j)}_r \right]\right)\,\d r\\
		&\hskip1in\ge q\mathop{\sum\sum}\limits_{1\le i\ne j\le k} \int_0^t
			\1_{B(0,\delta/\sqrt\varkappa)}\left(b^{(i)}_r - b^{(j)}_r \right)\,\d r\\
		&\hskip1in\ge q\mathop{\sum\sum}\limits_{1\le \ne j\le k} \int_0^t
			\1_{B(0,\delta/(2\sqrt\varkappa))}(b^{(i)}_r)
			\1_{B(0,\delta/2\sqrt\varkappa)}(b^{(j)}_r)\,\d r.
	\end{split}\end{equation}
	Recall Jensen's inequality,
	\begin{equation}\label{eq:Jensen}
		\E(\e^Z)\ge\e^{\E Z},
	\end{equation}
	valid for all nonnegative random variables $Z$.
	Because of \eqref{eq:Jensen}, Lemma \ref{lem:DC} and the preceding,
	we can conclude that
	\begin{equation}\begin{split}
		\E\left(|u_t(x)|^k\right) &\ge I^k\cdot \E\exp\left(
			q\mathop{\sum\sum}\limits_{1\le i\ne j\le k} \int_0^t
			\1_{B(0,\delta/(2\sqrt\varkappa))}(b^{(i)}_r)
			\1_{B(0,\delta/2\sqrt\varkappa)}(b^{(j)}_r)\,\d r\right)\\
		&= I^k\cdot \exp\left( qk(k-1)\cdot\int_0^t \left[ G\left(
			\frac{\delta}{2\sqrt\varkappa\sqrt{r}}\right)\right]^2
			\,\d r\right),
	\end{split}\end{equation}
	where $I:=\inf u_0$ and
	$G(z) := (2\pi)^{-d/2}\int_{\|x\|\le z} \e^{-\|x\|^2/2}
	\,\d x$ for all $z>0$.
	Because $k(k-1)\ge k^2/4$ for all $k\ge 2$, and we find that
	$\E ( | u_t(x)|^k )\ge I^k\cdot\exp (A_\delta k^2  )$,
	where $A_\delta$ is defined as
	\begin{equation}
		\frac{q}{4}
		\int_0^t\left[ G\left(\frac{\delta}{2\sqrt\varkappa\sqrt r}\right)\right]^2
		\,\d r=\inf_{x\in B(0,\delta)}f(x)\cdot
		\int_0^t\left[ \frac12G\left(
		\frac{\delta}{2\sqrt\varkappa\sqrt r}\right)\right]^2
		\,\d r.
	\end{equation}
	Finally, we observe that
	\begin{equation}
		0< \tilde{\Lambda}_d:= 
		\inf_{z>0}\left[\frac{\frac12G(z)}{1\wedge z^d}\right]^{\nicefrac12}<\infty.
	\end{equation}
	A few lines of computation yield the bound, $\sup_{\delta>0} A_{\delta} 
	\ge \tilde{\Lambda}_d a_t$. The lemma follows from this by readjusting
	and relabeling the constants.
\end{proof}

\section{Localization when $h\in L^2(\R^d)$ satisfies 
	\eqref{eq:h:a}}\label{local}
	
Throughout this section we assume that $h\in L^2(\R^d)$ is
nonnegative and satisfies  condition \eqref{eq:h:a}. Moreover, we let $u$ denote the solution to \textnormal{(SHE)}.
	
In order to simplify the notation we define, 
for every $x := (x_1,x_2,\ldots, x_d)\in \mathbf{R}^d$ 
and $a\in \mathbf{R}_{+}$,  
\begin{equation}
	[x-a\,,x+a] := [x_1-a\,,x_1+a]\times\cdots\times[x_d-a\,,x_d+a].
\end{equation}
That is, $[x-a\,,x+a]$ denotes the $\ell^\infty$ ball of radius $a$ around $x$.

Given an arbitrary $\beta>0$, 
define $U^{(\beta)}$ to be the solution to the random integral equation
\begin{align}\label{eq:U_beta}
	&U^{(\beta)}_t(x)  \\\notag
	&	= (p_t*u_0)(x)+\int_{(0,t)\times 
		[x-\beta\sqrt{t},x+\beta\sqrt{t}]} p_{t-s}(y-x)\sigma\left( 
		U^{(\beta)}_s(y) \right) F^{(h_{\beta})}(\d s\,\d y),
\end{align}
where $h_\beta$ is defined in \eqref{rho}. A comparison with
the mild form \eqref{mild} of the solution to \textnormal{(SHE)} shows
that $U^{(\beta)}$ is a kind of ``localized'' version of $u$.
Our goal is to prove that if $\beta$ is sufficiently large, then
$U^{(\beta)}_t(x)\approx u_t(x)$.

The method of
Dalang \cite{Dalang:99} can be used to prove that the predictable random
field $U^{(\beta)}$ exists, is unique up to a modification, and
satisfies the estimate
$\sup_{t\in[0,T]}\sup_{x\in\R^d}\E( |U^{(\beta)}_t(x)|^k)
<\infty$ for every $T>0$ and $k\ge 2$. Furthermore, the method of Foondun
and Khoshnevisan \cite{FK} shows that, in fact $U^{(\beta)}$ satisfies
a similar bound as does $u$ in \eqref{mom:upbd}. Namely,
 there exists a constant $D_1\in(0\,,\infty)$---depending on $\sigma$
 and $t$---such that for all $t>0$
 and $k\ge 2$,
\begin{equation}\label{mom:U_beta}
	\sup_{\beta>0}\sup_{x\in\R^d}
	\E\left( |U^{(\beta)}_t(x)|^k\right) \le D_1 \e^{D_1k^2t}.
\end{equation}
We skip the details of the proofs of these facts, as they require only
simple modifications to the methods of \cite{Dalang:99,FK}.

\begin{remark}\label{rem:NoKappa}
	We emphasize that $D_1$ depends only on $(t\,,f(0)\,,d,\sigma)$.
	In particular, it can be chosen to be independent of
	$\varkappa$. In fact, $D_1$ has exactly the same
	parameter dependencies as the upper bound for the
	moment estimate in \eqref{mom:upbd}; and the
	two assertions holds for very much the same reasons.\qed
\end{remark}

\begin{lemma}\label{lem:localization}
	For every $T>0$ there exists finite and positive constants 
	$G_{*}$ and $F_{*}$---depending only on 
	$(T\,,f(0)\,,d\,, \varkappa\,,b\,,\sigma)$---such 
	that for sufficiently large $\beta>0$ and $k\ge 1$,
	\begin{equation}
		\sup_{t\in [0,T]}\sup_{x \in \mathbf{R}^d} 
		\E\left(\left|u_t(x)-U_t^{(\beta)}(x)\right|^k\right) \le 
		\frac{G_*^k k^{k/2}\exp(F_* k^2 )}{\beta^{kb/2}},
	\end{equation}
	where $b\in(0\,,2)$ was introduced in Lemma \ref{lem:Omega}.
\end{lemma}

\begin{proof}
	By the triangle inequality,
	\begin{align}\label{u:decompose}
		&\left\| u_t(x)-U_t^{(\beta)}(x) \right\|_k\\\notag
		&\le
			\left\|u_t(x)-V_t^{(\beta)}(x) \right\|_k
			+\left\|V_t^{(\beta)}(x)-Y_t^{(\beta)}(x) \right\|_k
			+\left\|Y_t^{(\beta)}(x)-U_t^{(\beta)}(x) \right\|_k,
	\end{align}
	where
	\begin{equation}\label{eq:V_beta}
		V_t^{(\beta)}(x) := (p_t*u_0)(x)+\int_{(0,t)\times \R^d} p_{t-s}(y-x)
		\sigma\left(U_s^{(\beta)}(y)\right)F^{(h)}(\d s\,\d y),
	\end{equation}
	and
	\begin{equation}\label{eq:Y_beta}
		Y_t^{(\beta)}(x) := (p_t*u_0)(x)+\int_{(0,t)\times \R^d} p_{t-s}(y-x)
		\sigma\left( U_s^{(\beta)}(y)\right) F^{(h_{\beta})}(\d s\,\d y).
	\end{equation}
	In accord with \eqref{eq3.24} and \eqref{mom:U_beta},
	\begin{align}\begin{split}\label{eq:V-Y}
			\|V^{(\beta)}-Y^{(\beta)}\|_k\le \text{const}\cdot \sqrt{\frac{kt}{\beta^b}}\e^{D_1tk}
	\end{split}\end{align}
	where we remind that $D_1$ is a constant that does not depend on $\varkappa$. 
	Next we bound the quantity $\|Y^{(\beta)}-U^{(\beta)}\|_k$,
	using the  Burkholder--Davis--Gundy inequality,
	\textnormal{(BDG)} and obtain the following:
	\begin{equation}\begin{split}
		&\left\|Y_t^{(\beta)}(x)-U_t^{(\beta)}(x) \right\|_k\\
		&=\left\| \int_{(0,t)\times [x-\beta\sqrt{t},x+\beta\sqrt{t}]^c} p_{t-s}(y-x)
			\sigma\left( U^{(\beta)}_s(y) \right) F^{(h_{\beta})}(\d s\,\d y)
			\right\|_k\\
		&\le \textnormal{const}\cdot
			\sqrt{kf(0)} \left(  \int_0^t\d s
			\int_{[x-\beta\sqrt{t},x+\beta\sqrt{t}]^c}\d y
			\int_{[x-\beta\sqrt{t},x+\beta\sqrt{t}]^c}\d z\ \mathcal{W}\right)^{\nicefrac12},
	\end{split}\end{equation}
	where
	\begin{equation}
		\mathcal{W} := 
		p_{t-s}(y-x)p_{t-s}(z-x)
		\left( 1+\left\| U^{(\beta)}_s(y)\right\|_k\right)
		\left(1+\left\| U^{(\beta)}_s(z)\right\|_k\right). 
	\end{equation}
	Therefore, \eqref{mom:U_beta} implies that
	\begin{equation}
		\left\|Y_t^{(\beta)}(x)-U_t^{(\beta)}(x) \right\|_k
		\le  D_2\e^{D_2tk} \sqrt{kf(0) \cdot \tilde{\mathcal{W}}},
	\label{eq:Y-U}
	\end{equation}
	where $D_2\in(0\,,\infty)$ depends only on $d$, $f(0)$,
	and $t$, and
	\begin{equation}
		\tilde{\mathcal{W}} := 
		\int_0^t\d s\left(\int_{[x-\beta\sqrt{t},\,x+\beta\sqrt{t}]^c}
		\d y\ p_{t-s}(y-x)\right)^2.
	\end{equation}
	Before we proceed further, let us note that
	\begin{equation}\label{eqp}
		\int_{\substack{z\in\R:\\ |z|>\beta\sqrt{t} }}
		\frac{\e^{-z^2/(2\varkappa(t-s))}}{\sqrt{2\pi \varkappa(t-s)}}
		\,\d z
		\le 2\cdot \exp\left(-
		\frac{\beta^2 t}{4\varkappa(t-s)}\right).
	\end{equation}
	Using the above in \eqref{eq:Y-U}, we obtain
	\begin{equation}\label{eq2:Y-U}
		\left\|Y_t^{(\beta)}(x)-U_t^{(\beta)}(x) \right\|_k 
		\le 2 D_2\e^{D_2tk} \sqrt{ktf(0)}\exp\left(
		-\frac{d\beta^2 }{4\varkappa}\right).
	\end{equation}

	Next we estimate $\|u_t(x)-V_t^{(\beta)}(x)\|_k$.  
	An application of \textnormal{(BDG)} yields
	\begin{align}
		& \left\|u_t(x)-V_t^{(\beta)}(x) \right\|_k\\\notag
		&\le \left\|\int_{(0,t)\times \R^d} p_{t-s}(y-x)
			\left\lbrace \sigma(u_s(y) )-\sigma(U_s^{(\beta)}(y))
			\right\rbrace F^{(h)}(\d s\,\d y)\right\|_k \\\notag
		&\le 2\sqrt k \left\| \int_0^t\d s \int_{\R^d}\d y
			\int_{\R^d}\d z\ f(y-z) p_{t-s}(y-x)p_{t-s}(z-x)
			\mathcal{Q}\right\|_{k/2},
	\end{align}
	where $\mathcal{Q} := | \sigma (u_s(y) )-\sigma(U_s^{(\beta)}(y))
	|\cdot|\sigma (u_s(z) )-\sigma (U_s^{(\beta)}(z) )|.$
	Since $\sigma$ is Lipschitz continuous, it follows from Minkowski's
	inequality that
	\begin{equation}\label{eq:u-V}
		\left\|u_t(x)-V_t^{(\beta)}(x) \right\|_k^2
		\le 4\text{\rm Lip}_\sigma^2 kf(0) \int_0^t
		\mathcal{Q}^*_s\,\d s,
	\end{equation}
	where $\mathcal{Q}^*_s := \sup_{y\in\R^d} \|
	u_s(y) -U_s^{(\beta)}(y) \|_k^2.$
	 Equations \eqref{u:decompose}, \eqref{eq:V-Y} and \eqref{eq2:Y-U} together imply that
	$\mathcal{Q}^*_t
	\le \textnormal{const}\cdot kt\beta^{-b}\e^{\textnormal{const}
	\cdot kt}+\textnormal{const}\cdot kf(0)
	\cdot \int_0^t \mathcal{Q}^*_s\, \d s.$
	 Therefore, 
	\begin{equation}
		\mathcal{Q}^*_t \le 
		\textnormal{const}\cdot\left( \frac{tk
		\e^{\textnormal{const}\cdot k t}}{\beta^b} \right)
		\qquad\text{for all $t>0$},
	\end{equation}
	owing to Gronwall's inequality.
	Because ``const'' does not depend on $(k\,,t)$,
	we take both sides to the power $k/2$ in order to finish the proof.
 \end{proof}  
 
 Now, let us define $U_t^{(\beta, n)}$ to be the $n$th 
 Picard-iteration approximation of $U_t^{(\beta)}(x)$. 
 That is , $U_t^{(\beta,0)}(x):=u_0(x)$ , and for all $l\ge 0$,
 \begin{equation}\begin{split}\label{ubn}
	 &U_t^{(\beta, l+1)}(x)\\
	 &:= \left(p_t*u_0\right)(x)+\int_{(0,t)\times [x-\beta\sqrt{t}, x+\beta\sqrt{ t}]}
	 	p_{t-s}(y-x) \sigma\left(U_s^{(\beta, l)}(y)\right)\,
		F^{(h_{\beta})}(\d s\, \d y) .
 \end{split}\end{equation}
 
\begin{lemma}\label{lem:u-Ubl}
For every $T>0$ there exists finite and positive constants 
	$G$ and $F$---depending only on 
	$(T\,,f(0)\,,d\,, \varkappa\,,b\,,\sigma)$---such 
	that for sufficiently large $\beta>0$ and $k\ge 1$,
	\begin{equation}
		\sup_{t\in [0,T]}\sup_{x \in \mathbf{R}^d} 
		\E\left(\left|u_t(x)-U_t^{(\beta, [\log \beta]+1)}(x)\right|^k\right) \le 
		\frac{G^k k^{k/2}\exp(F k^2 )}{\beta^{kb/2}},
	\end{equation}
	where $b\in(0\,,2)$ was introduced in Lemma \ref{lem:Omega}.
\end{lemma}

\begin{proof} 
	The method of Foondun and Khoshnevisan \cite{FK:TRAMS} 
	can be used to show that if $\delta:=D'k$ for a sufficiently-large 
	positive and finite constant $D'$, then 
	\begin{equation}\label{eq:U-U}
		\mathcal{M}_{\delta}^{(k)} \left(U^{(\beta)}-U^{(\beta,n)}\right) 
		\le \text{const}\cdot \e^{-n} \qquad\text{for all 
		$n\ge 0$ and $k \in [2\,,\infty)$.}
	\end{equation}
	To elaborate, we replace the $u^n$ of Ref.\ \cite[(5.36)]{FK:TRAMS} 
	by our $U^{(\beta, n)}$ and obtain
	\begin{equation}
		\|U^{(\beta, n+1)}-U^{(\beta, n)}\|_{k,\theta} \le 
		\|U^{(\beta, n)}-U^{(\beta, n-1)}\|_{k,\theta} \cdot Q(k\,,\theta),
	\end{equation}
	where $\|X\|_{k,\theta} := \{ \sup_{t\ge 0}\sup_{x\in \mathbf{R}} 
	\e^{-\theta t} \E (|X_t(x)|^k ) \}^{1/k}
	=\mathcal{M}_{\theta/k}^{(k)}(X),$
	for all  random fields $\{X_t(x)\}_{t>0,x\in\R^d}$,
	and $Q(k\,,\theta)$ is defined in Theorem 1.3 of 
	\cite{FK:TRAMS}. We recall from \cite{FK:TRAMS} that $Q(k\,,\theta)$ satisfies
	the following bounds:
	\begin{equation}
		Q(k\,,\theta)\le 
		\sqrt{4k\lip^2\cdot\Upsilon\left(\frac{2\theta}{k}\right)} 
		\le \text{const}\cdot \frac{k\|h\|_{L^2(\mathbf{R}^d)}}{\theta^{\nicefrac12}}.
	\end{equation}
	[The function $\Upsilon$ is defined in \cite[(1.8)]{FK:TRAMS}.]
	Therefore, it follows readily from these bounds that if
	$\theta:=D''k^2$ for a large enough $D''>0$, then
	\begin{equation}
		\|U^{(\beta, n+1)}-U^{(\beta, n)}\|_{k,\theta} 
		\le \e^{-1}\|U^{(\beta, n)}-U^{(\beta, n-1)}\|_{k,\theta}.
	\end{equation}
	We obtain \eqref{eq:U-U} from this inequality. 
	
	Finally we set $n := [\log \beta]+1$ and apply the preceding 
	together with Lemma \ref{lem:localization} to finish the proof.
\end{proof}
	
For every $x,\,y\in \R^d$, let us define
\begin{equation}\label{d}
	D(x\,,y):=\min_{1\le l\le d} |x_l-y_l|.
\end{equation} 
	
\begin{lemma} \label{lem:un-ind}
	Choose and fix $\beta \ge 1,\, t>0$ and let $n := [\log \beta]+1$. 
	Also fix $x^{(1)},x^{(2)},\cdots \in \mathbf{R}^d$ such that 
	$D(x^{(i)}\,, x^{(j)})\ge 2n\beta(1+\sqrt{t})$. 
	Then $\{ U_t^{(\beta, n)}(x^{(j)})\}_{j\in \Z}$ are independent random variables. 
\end{lemma}
	
\begin{proof}
	The lemma follows from the recursive definition of the 
	$U^{(\beta,n)}$'s. Indeed, $U^{(\beta,n)}_t(x)$ depends on 
	$U^{(\beta, n-1)}_s(y),\,y\in [x-\beta\sqrt{t}\,,x+\beta\sqrt{t}],\,s\in [0,t]$. 
	An induction argument shows that  $U^{(\beta,n)}_t(x)$ depends only
	on the values of $U^{(\beta,1)}_s(y)$, as $y$ varies in
	$[x-(n-1)\beta\sqrt{t}\,,x+(n-1)\beta\sqrt{t}]$ and $s$ in $[0\,,t]$. 
		
		Finally, we observe that $\{U_s^{(\beta, 1)}(x)\}_{%
		s\in[0,t],\,  x\in \mathbf{R}^d}$ is a Gaussian random field 
		that has the property that $U_s^{(\beta, 1)}(x)$ and 
		$U_s^{(\beta,1)}(x')$ are independent whenever 
		$D(x\,, x') \ge 2\beta(1+\sqrt{t})$. 
		[This assertion follows from a direct covariance calculation in
		conjunction with the fact that $(h_\beta*\tilde{h}_\beta)(z)=0$
		when $D(0,z) \ge 2\beta$].
\end{proof}

\section{Proof of Theorem \ref{th:boundedness}}

In this section we prove our first main theorem (Theorem \ref{th:boundedness}).
It is our first proof primarily because the following derivation is
the least technical and requires that we keep track of very few
parameter dependencies in our inequalities. 

Define for all $k\in[2\,,\infty)$, $\beta>0$, and predictable random
fields $Z$,
\begin{equation}
	\mathcal{Y}_\beta^{(k)}(Z) := 
	\sup_{\substack{t>0\\x\in\R^d}} \left[
	\exp\left( -\beta t+\sqrt{\frac{\beta}{8\varkappa}}\ \|x\|\right)
	\cdot \left\| Z_t(x) \right\|_k \right].
\end{equation}

Let us begin by developing a weighted Young's inequality for
stochastic convolutions. This is similar in spirit to the results of
Conus and Khoshnevisan \cite{ConusK}, extended to the present
setting of correlated noise. However, entirely new ideas
are needed in order to develop this result; therefore, we include a complete proof.

\begin{proposition}[A weighted stochastic Young inequality]\label{pr:Young}
	Let $Z:=\{Z_t(x)\}_{t>0,x\in\R^d}$ be a predictable random field.
	Then for all real numbers $k\in[2\,,\infty)$
	and $\beta>0$,
	\begin{equation}
		\mathcal{Y}_\beta^{(k)}
		\left( p*Z\dot{F}\right) \le 
		\mathcal{Y}_\beta^{(k)}(Z)\cdot
		\sqrt{2^{d} k  ( R_{\beta/4}f  ) (0)},
	\end{equation}
	where $R_{\beta}$ is the resolvent operator defined in \eqref{eq:R_beta}.
\end{proposition}

\begin{proof}
	For the sake of typographical ease we write 
	$c=c(\beta) :=\sqrt{\beta/(8\varkappa)}$ throughout the proof.
	
	Our derivation of \eqref{4.24} yields the following estimate:
	\begin{equation}\begin{split}
		&\left\| \left( p*Z\dot{F} \right)_t(x) \right\|_k^2 \\
		&\hskip.5in\le 4k\int_0^t\d s\int_{\R^d}\d y\int_{\R^d}\d z\
			f(y-z) p_{t-s}(y-x)p_{t-s}(z-x)\cdot\mathcal{Z},
	\end{split}\end{equation}
	where $\mathcal{Z} := \| Z_s(y)\cdot Z_s(z) \|_{k/2}\le
	\|Z_s(y)\|_k\cdot\|Z_s(z)\|_k$.
	Consequently, for all $\beta>0$,
	\begin{align}
		&\left\| \left( p*Z\dot{F} \right)_t(x) \right\|_k^2 \\\notag
		&\le 4k\left[ \mathcal{Y}_\beta^{(k)}(Z) \right]^2\cdot
			\int_0^t\d s\int_{\R^d}\d y\int_{\R^d}\d z\
			f(y-z) P_{s}(y\,,y-x) P_{s}(z\,,z-x),
	\end{align}
	where  $P_s(a\,,b) := \e^{\beta s-c\|a\|} p_{t-s}(b)$
	for all $s>0$ and $a\in\R^d$.
	Since $\|y\|\ge\|x\|-\|x-y\|$ and $\|z\|\ge\|x\|-\|x-z\|$,
	it follows that
	\begin{align}\label{eq:QQ}
		&\hskip-.3in\left\| \left( p*Z\dot{F} \right)_t(x) \right\|_k^2 \\\notag
		&\hskip.6in\le 4k\e^{2\beta t-2c\|x\|}
			\left[ \mathcal{Y}_\beta^{(k)}(Z) \right]^2\cdot
			\int_0^\infty \e^{-\beta s}
			\left( Q_s*Q_s*f\right)(0)\,\d s,
	\end{align}
	where 
	\begin{equation}\label{eq:Q_s}
		Q_s(a) :=  \e^{-(\beta s/2)+c\|a\|}p_s(a)
		\qquad\text{for all $s>0$ and $a\in\R^d$}.
	\end{equation}
	Clearly,
	\begin{equation}\label{eq:Q1}
		\text{if $\frac{\beta s}{2} \ge c\|a\|$, then}\quad
		Q_s(a) \le p_s(a).
	\end{equation}
	Now consider the case that $(\beta s/2)< c\|a\|$. Then,
	\begin{equation}
		c\|a\| - \frac{\|a\|^2}{2s\varkappa}
		= -\frac{\|a\|^2}{2s\varkappa} 
		\left( 1- \frac{2s\varkappa c}{\|a\|}\right)
		< -\frac{\|a\|^2}{2s\varkappa} 
		\left( 1  - \frac{4\varkappa c^2}{\beta} \right)=
		-\frac{\|a\|^2}{4s\varkappa} .
	\end{equation}
	We can exponentiate the preceding to see that, in the case that
	$(\beta s/2)< c\|a\|$,
	\begin{equation}\label{eq:Q2}
		Q_s(a) \le 
		\frac{\e^{-(\beta s/2) - \|a\|^2/(4s\varkappa)}}{(2\pi\varkappa s)^{d/2}}
		\le 2^{d/2} p_{2s}(a).
	\end{equation}
	Since $p_s(a) \le 2^{d/2} p_{2s}(a)$ for all $s>0$ and $a\in\R^d$, 
	we deduce from \eqref{eq:Q1} and \eqref{eq:Q2} that \eqref{eq:Q2}
	holds for all $s>0$ and $a\in\R^d$.
	Therefore, the Chapman--Kolmogorov equation implies that
	$Q_s*Q_s \le 2^d p_{4s}$, and hence
	\begin{equation}\begin{split}
		\int_0^\infty \e^{-\beta s} 
			\left( Q_s*Q_s*f\right)(0)\,\d s &\le
			2^d\int_0^\infty \e^{-\beta s} \left( 
			p_{4s}*f\right)(0)\,\d s\\
		&=  2^{d-2} ( R_{\beta/4}f)(0).
	\end{split}\end{equation}
	The proposition now follows from \eqref{eq:QQ}.
\end{proof}

Next we state and prove an elementary 
estimate for the heat semigroup.

\begin{lemma}\label{lem:pt:phi}
	Suppose $\phi:\R^d\to\R$ is a measurable function
	and $L(c):=\sup_{x\in\R^d} 
	(\e^{c\|x\|}|\phi(x)|)$ is finite
	for some $c>0$. Then 
	$\mathcal{Y}_{8c^2\varkappa}^{(k)} ( p*\phi  )\le 2^{d/2}L(c)$
	for all $k\in[2\,,\infty)$.
\end{lemma}

\begin{proof}
	Let us define $\beta:=8c^2\varkappa$, so that 
	$c =\sqrt{\beta/(8\varkappa)}$. Then,
	\begin{align}\notag
		\e^{-\beta t+c\|x\|} \left| (p_t*\phi)(x) \right|
			&= \int_{\R^d}\e^{-\beta t+ c\|x\|}
			p_t(x-y) |\phi(y)|\,\d y\\
		&\le\int_{\R^d} \e^{-\beta t+c\|x-y\|}
			p_t(x-y) \cdot \e^{c\|y\|}|\phi(y)|\,\d y\\\notag
		&\le L(c)\int_{\R^d}\e^{-\beta t+c\|z\|}
			p_t(z)\,\d z \le L(c)\int_{\R^d} Q_t(z)\,\d z,
	\end{align}
	where the function $Q_t(z)$ is defined in \eqref{eq:Q_s}.
	We apply \eqref{eq:Q2} to deduce from this that
	$\e^{-\beta t+c\|x\|} | (p_t*\phi)(x) |
	\le 2^{d/2}L(c)\int_{\R^d}p_{2t}(z)\,\d z=2^{d/2}L(c).$
	Optimize over $t$ and $x$ to finish.
\end{proof}

We will next see how to combine the preceding results  in order
to establish the rapid decay of the moments of the solution to
\textnormal{(SHE)} as $\|x\|\to\infty$.

\begin{proposition}\label{pr:moment:decay}
	Recall that $u_0:\R^d\to\R$ is a bounded and measurable function and $\sigma(0)=0$.
	If, in addition, $\limsup_{\|x\|\to\infty} \|x\|^{-1}\log |u_0(x)|=-\infty$,
	then
	\begin{equation}
		\limsup_{\|x\|\to\infty}  \frac{\log\E
		(|u_t(x)|^k)}{\|x\|}<0\quad\text{for all $t>0$ and $k\in[2\,,\infty)$.}
	\end{equation}
\end{proposition}

\begin{proof}
	For all $t>0$ and $x\in\R^d$, define $u^{(0)}_t(x) := u_0(x)$, and
	\begin{equation}
		u^{(l+1)}_t(x) := (p_t*u_0)(x) + \left( p * \left(\sigma\circ
		u^{(l)}\right)\dot{F} \right)_t(x)
		\quad\text{for all $l\ge 0$}.
	\end{equation}
	That is, $u^{(l)}$ is the $l^{\mbox{\scriptsize th}}$ level  in the Picard iteration
	approximation to the solution $u$. By the triangle inequality,
	\begin{equation}\begin{split}
		\mathcal{Y}_\beta^{(k)}\left( u^{(l+1)}\right) &\le
			 \mathcal{Y}_\beta^{(k)}\left(  p*u_0 \right) + 
			 \mathcal{Y}_\beta^{(k)}\left( \left( p* \left( \sigma\circ u^{(l)}
			 \right)\dot{F}\right)\right)\\
		&\le \mathcal{Y}_\beta^{(k)}\left( p*u_0\right)+
			\mathcal{Y}_\beta^{(k)}\left( \sigma\circ u^{(l)}\right)
			\cdot\sqrt{2^{d}k(R_{\beta/4}f)(0)};
	\end{split}\end{equation}
	see Proposition \ref{pr:Young}. Because $|\sigma(z)|\le \lip|z|$
	for all $z\in\R^d$, it follows from the triangle inequality that
	\begin{equation}
		\mathcal{Y}_\beta^{(k)}\left( u^{(l+1)}\right) \le
		\mathcal{Y}_\beta^{(k)}\left( p*u_0\right) + 
		\mathcal{Y}_\beta^{(k)}\left( u^{(l)}\right)\cdot
		\sqrt{2^{d}\lip^2 k (R_{\beta/4}f)(0)}.
	\end{equation}
	By the dominated convergence theorem, $\lim_{q\to\infty}(R_qf)(0)=0$.
	Therefore, we may choose $\beta$ large enough to ensure that
	the coefficient of $\mathcal{Y}_\beta^{(k)}(u^{(l)})$ in the preceding is
	at most $\nicefrac12$. The following holds for this choice of $\beta$:
	\begin{equation}
		\sup_{l\ge 0}\mathcal{Y}_\beta^{(k)}\left( u^{(l+1)}\right) \le 2
		\mathcal{Y}_\beta^{(k)}\left(  p*u_0\right)
		\le 2^{(d+2)/2}\sup_{x\in\R^d}
		\left( \e^{\|x\|\sqrt{\beta/(8\varkappa)}} |u_0(x)|\right);
	\end{equation}
	we have applied Lemma \ref{lem:pt:phi} in order to deduce the
	final inequality. According to the theory of Dalang \cite{Dalang:99},
	$u^{(l)}_t(x)\to u_t(x)$ in probability as $l\to\infty$,
	for all $t>0$ and $x\in\R^d$. Therefore, Fatou's lemma implies 
	that
	\begin{equation}
		\mathcal{Y}_\beta^{(k)} ( u) \le 2^{(d+2)/2}\sup_{x\in\R^d}
		\left( \e^{\|x\|\sqrt{\beta/(8\varkappa)}} |u_0(x)|\right);
	\end{equation}
	whence follows the result [after some arithmetic].
\end{proof}

Next we introduce a fairly crude estimate for the spatial oscillations
of the solution to \textnormal{(SHE)}, in the sense of $L^k(\P)$. 
We begin with an estimate of $L^1(\R^d)$-derivatives of
the heat kernel. This is without doubt a well-known result,
though we could not find an explicit reference.
In any event, the proof is both elementary and short; therefore
we include it for the sake of completeness.

\begin{lemma}\label{p_bound}
	For all $s>0$ and $x\in \R^d$, 
	\begin{equation}
		\int_{\R^d}\left| p_s(y-x)-p_s(y)\right|\,\d y
		\le \textnormal{const}\cdot
		\left(\frac{\|x\|}{\sqrt{\varkappa s}}\wedge 1\right),
	\end{equation}
	where the implied constant does not depend on
	$(s\,,x)$.
\end{lemma}

\begin{proof}
	For $s$ fixed,
	let us define 
	\begin{equation}
		\mu_d(r) = \mu_d(r\,;s) := \sup_{\substack{z\in\R^d\\ \|z\|\le r}}
		\int_{\R^d} | p_s(y-z)-p_s(y)  |\,\d y
		\quad\text{ for all $r>0$}.
	\end{equation}
	First consider the case that $d=1$. In that case, we may use the
	 differential equation $p'_s(w)=-(w/\varkappa s)p_s(w)$ in order to see that
	\begin{equation}\begin{split}
		\mu_1(|x|) &= \sup_{z\in(0,|x|)} \int_{-\infty}^\infty
			\left| \int_{y-z}^y p_s' (w)\,\d w\right|\,\d y\\
		&\le \frac{1}{\varkappa s} \sup_{z\in(0,|x|)} \int_{-\infty}^\infty\d y
			\int_{y-z}^y \d w\ |w| p_s (w) =
			\frac{|x|}{\varkappa s}\int_{-\infty}^\infty |w|p_s(w)\,\d w\\
		&=\sqrt{\frac{2}{\pi \varkappa s}}\,|x|\hskip1in \text{ for all $x\in\R$}.
	\end{split}\end{equation}
	For general $d$, we can integrate one coordinate at a time and then apply
	the triangle inequality to see that for all $x:=(x_1\,,\ldots,x_d)\in\R^d$,
	$\mu_d(\|x\|)\le\sum_{j=1}^d\mu_1(\|x\|)
	\le \sqrt{2/(\pi \varkappa s)}\,d\|x\|$.
	Because $|p_s(y-x)-p_s(y)|\le p_s(y-x)+p_s(y)$, we also have
	 $\mu_d(\|x\|)\le 2$.
\end{proof}

\begin{proposition}\label{pr:mod:crude}
	Let us assume that: (i) $\limsup_{\|x\|\to\infty}
	\|x\|^{-1}\log|u_0(x)|=-\infty$, (ii) $\sigma(0)=0$, and (iii) 
	$\int_0^1 s^{-a} (p_s*f)(0)\,\d s<\infty$ for some $a\in(0\,,\nicefrac12)$.
	Then for all $t>0$ and $k\in[2\,,\infty)$ there exists
	a constant  $C\in(1\,,\infty)$ such that uniformly for all $x,x'\in\R^d$
	that satisfy $\|x-x'\|\le 1$,
	\begin{equation}
		\E \left( |u_t(x)-u_t(x')|^k \right) \le C \exp\left(
		-\frac{\|x\|\wedge\|x'\|}{C}\right)\cdot\|x-x'\|^{ak/4}.
	\end{equation}
\end{proposition}

\begin{proof}
	First of all, we note that
	\begin{align}\notag
		\left| (p_t*u_0)(x) - (p_t*u_0)(x') \right|
			&\le \|u_0\|_{L^\infty(\R^d)}\cdot
			\int_{\R^d} \left| p_t(y-x)-p_t(y-x')\right|
			\,\d y\\
		&\le\textnormal{const}\cdot\|x-x'\|; \label{eq:cont_ic}
	\end{align}
	see Lemma  \ref{p_bound}. 
	Now we may use this estimate and
	the same argument that led us to \eqref{4.24} in order
	to deduce that for all $\ell\in[2\,,\infty)$,
	\begin{align}
		&\| u_t(x)-u_t(x')\|_\ell^2 
			\le\textnormal{const}\cdot \|x-x'\|^2\\\notag
		&\hskip1.2in+
			\textnormal{const}\cdot \int_0^t\d s\int_{\R^d}\d y\int_{\R^d}
			\d z\ f(y-z)\,\mathcal{A}\,\mathcal{B}_s(y)\,\mathcal{B}_s(z),
	\end{align}
	where
	\begin{equation}\begin{split}
		&\mathcal{A} :=
			\mathcal{A}_s(y\,,z) := \left\| u_s(y)\cdot u_s(z) \right\|_{\ell/2}^2,
			\qquad\text{and}\\
		&\mathcal{B}_s(w) := \left|  p_{t-s}(w-x) - p_{t-s}(w-x') \right|
			\qquad\text{for all $w\in\R^d$}.
	\end{split}\end{equation}
	According to \cite{Dalang:99}, $\sup_{s\in[0,T]}\sup_{y,z\in\R^d}
	\mathcal{A}<\infty$. On the other hand,
	$(\mathcal{B}_s*f)(z) \le 2\sup_{w\in\R^d} (p_{t-s}*f)(w)$, and the latter
	quantity is equal to $2(p_{t-s}*f)(0)$ since $p_r*f$ is positive definite and
	continuous for all $r>0$ [whence is maximized at the origin]. We can
	summarize our efforts as follows:
	\begin{align}
		&\|u_t(x)-u_t(x') \|_\ell^2\\\notag
		&\le \textnormal{const}\cdot \|x-x'\|^2
			+ \textnormal{const}\cdot \int_0^t(p_s*f)(0)\,\d s 
			\int_{\R^d}\d z\ \left|  p_s(z-x) - p_s(z-x') \right|\\\notag
		&\le \textnormal{const}\cdot \|x-x'\|^2
			+\textnormal{const}\cdot \int_0^t  (p_s*f)(0)\left(
			\frac{\|x-x'\|}{\sqrt s}\wedge 1\right)\,\d s;
	\end{align}
	see Lemma \ref{p_bound} below, for instance. We remark that the  implied
	constants do not depend on $(x\,,x')$. Since
	$r\wedge 1 \le r^{2a}$ for all $r>0$, it follows that
	\begin{equation}\label{eq:u:ell}
		\|u_t(x)-u_t(x')\|_\ell\le \textnormal{const}\cdot  \|x-x'\|^{a/2},
	\end{equation}
	where the implied constant does not depend on $(x\,,x')$ as long
	as $\|x-x'\|\le 1$ [say].
	Next we write
	\begin{equation}\begin{split}
		&\E\left( |u_t(x)-u_t(x')|^k\right)\\
		& \hskip.5in\le \E\left( |u_t(x)-u_t(x')|^{k/2}\cdot 
			\left\{|u_t(x)|+|u_t(x')|\right\}^{k/2}\right)\\
		& \hskip.5in\le \textnormal{const}\cdot
			\left\| u_t(x)-u_t(x') \right\|_k^{k/2} \left(\|u_t(x)\|_k^{k/2}
			\vee \|u_t(x')\|_k^{k/2} \right),
	\end{split}\end{equation}
	by H\"older's inequality. Proposition \ref{pr:moment:decay}
	and Eq.\ \eqref{eq:u:ell} together complete our proof.
\end{proof}

Proposition \ref{pr:mod:crude} and a quantitative form of Kolmogorov's
continuity lemma \cite[pp.\ 10--12]{Minicourse} readily imply the following.

\begin{corollary}\label{co:mod:crude}
	Let us assume that: (i) $\limsup_{\|x\|\to\infty}
	\|x\|^{-1}\log|u_0(x)|=-\infty$, (ii) $\sigma(0)=0$, and (iii) 
	$\int_0^1 s^{-a} (p_s*f)(0)\,\d s<\infty$ for some $a\in(0\,,\nicefrac12)$.
	Then for all $t>0$ and $k\in[2\,,\infty)$ there exists
	a constant  $C\in(1\,,\infty)$ such that uniformly for all hypercubes
	$T\subset\R^d$ of sidelength $2/\sqrt{d}$,
	\begin{equation}
		\E \left( \sup_{x,x'\in T}|u_t(x)-u_t(x')|^k \right) \le C \exp\left(
		-\frac1C \inf_{z\in T}\|z\|\right).
	\end{equation}
\end{corollary}

Finally, we are in position to establish Theorem \ref{th:boundedness}.

\begin{proof}[Proof of Theorem \ref{th:boundedness}]
	Define
	\begin{equation}
		T(x) := \left\{ y\in\R^d:\, \max_{1\le j\le d}
		|x_j-y_j|\le \frac{2}{\sqrt{d}}\right\}\quad
		\text{for every $x\in\R^d$}.
	\end{equation}
	Then, for all $t>0$ and $k\in[2\,,\infty)$,
	there exists a constant $c\in(0\,,1)$ such that
	uniformly for every $x\in\R^d$,
	\begin{align}\notag
		\E\left( \sup_{y\in T(x)}|u_t(y)|^k\right)
			&\le 2^{k}\left\{ \E\left( |u_t(x)|^k\right)
			+\E\left(\sup_{y\in T(x)}
			|u_t(y)-u_t(x)|^k\right)\right\}\\
		&\le\frac 1c\cdot\left\{ \e^{-c\|x\|} + 
			\exp\left( -c \inf_{y\in T(x)}\|y\|\right)\right\};
	\end{align}
	see Proposition \ref{pr:moment:decay} and Corollary 
	\ref{co:mod:crude}. Because $\inf_{y\in T(x)}\|y\|\ge\|x\|-1$
	for all $x\in\Z^d$, the preceding is
	bounded by $\textnormal{const}\cdot\exp(-\textnormal{const}
	\cdot\|x\|)$, whence $\E( \sup_{z\in\R^d}|u_t(z)|^k)
	\le \sum_{x\in\Z^d} \E( \sup_{y\in T(x)}|u_t(y)|^k)$ is finite.
\end{proof}

\section{Proof of Theorem \ref{th:sigma:bdd}}
Throughout this section, we assume
that $f=h*\tilde{h}$ for some nonnegative function $h\in L^2(\R^d)$
that satisfies \eqref{eq:h:a}. Moreover, we let $u$ denote the solution to \textnormal{(SHE)}.

\subsection{The first part}

Here and throughout we define for all $R,t>0$
\begin{equation}\label{eq:u*}
	u_t^{*}(R) := \sup_{\|x\|\le R}|u_t(x)|.
\end{equation}
As it turns out, it is easier to prove slightly stronger statements
than \eqref{eq:sigma:bdd1} and \eqref{eq:sigma:bdd2}. 
The following is the stronger version of \eqref{eq:sigma:bdd1}.

\begin{proposition}\label{prop2}
	If $\sigma $ is bounded uniformly  away from zero, then
	\begin{equation}
		\liminf_{R\to \infty}\frac{u_t^{*}(R) }{(\log R)^{\nicefrac14}}>0 
		\qquad\mbox{a.s.}
	\end{equation}
\end{proposition}

\begin{proof}
	Let us introduce a free parameter $N\ge1$, which is an integer
	that we will select carefully later on in the proof. 
	
	As before, let us denote $n=[\log \beta]+1$. For all 
	$\theta,R>0$ and $x^{(1)},x^{(2)},\cdots,x^{(N)} \in \R^d$, we may write
	\begin{align}\notag
		&\P\left\{\max_{1\le j\le N} |u_t(x^{(j)})| < 
			\theta (\log R)^{\nicefrac14}\right\} 
			\le \P\left\{\max_{1\le j\le N} |U^{(\beta,n)}_t(x^{(j)})| < 
			2\theta (\log R)^{\nicefrac14} \right\}\\
		&\hskip1.3in+ 
			\P\left\{\max_{1\le j\le N} |u_t(x^{(j)})-U^{(\beta,n)}_t(x^{(j)})| >\theta 
			(\log R)^{\nicefrac14}\right\}.\label{eq0}
	\end{align}
	We bound these quantities in order.
	
	Suppose in addition that $D(x^{(i)}\,,x^{(j)}) \ge 2n\beta(1+\sqrt{t})$ 
	whenever $i\ne j$, where $D(x\,,y)$ was defined in \eqref{d}. 
	Because of Lemma \ref{lem:un-ind}, the collection
	$\{U^{(\beta,n)}_t(x_j)\}_{j=1}^N$ is comprised of 
	independent random variables. 
	Consequently,
	\begin{equation}\begin{split}
		&\P\left\{ \max_{1\le j\le N}|U^{(\beta,n)}_t(x^{(j)})| 
			<2\theta(\log R)^{\nicefrac14}\right\}\\
		&\hskip.5in\le \left(
			\P\left\{ \left\vert U^{(\beta,n)}_t(x^{(1)})\right\vert <2\theta(\log R)^{\nicefrac14}
			\right\}\right)^N
			\le \left(\mathcal{T}_1 + \mathcal{T}_2\right)^N,
	\end{split}\end{equation}
	where
	\begin{equation}\begin{split}
		\mathcal{T}_1 &:= \sup_{x\in\R^d}
			\P\left\{\left\vert u_t(x)\right\vert < 3\theta(\log R)^{\nicefrac14}
			\right\},\\
		\mathcal{T}_2 &:= \sup_{x \in\R^d}
			\P\left\{ | u_t(x)- U^{(\beta,n)}_t(x)| >
			\theta(\log R)^{\nicefrac14}\right\}.
	\end{split}\end{equation}
	According to Lemma \ref{tail_bound0},
	$\mathcal{T}_1 \le 1 - a(\varkappa)^{-\frac12}R^{-2(3\theta)^4 a(\varkappa)}$
	for all $R$ sufficiently large;
	and Lemma \ref{lem:u-Ubl} implies that there exists a finite
	constant $m\ge 1$ such that uniformly for all $k,\beta\ge m$,
	$\mathcal{T}_2 \le G^k k^{k/2}\e^{Fk^2}/(\theta^k
	\beta^{kb/2}(\log R)^{k/4}) \le c_1(k) \beta^{-kb/2}(\log R)^{-k/4}$
	for a finite and positive constant $c_1(k):=c_1(k\,,G\,,F\,,\theta)$.
	We combine the preceding to find that
	\begin{equation}
		\left( \mathcal{T}_1+\mathcal{T}_2\right)^N
		\le \left( 1 - \frac{a(\varkappa)^{-\frac12}}{R^{2(3\theta)^4 a(\varkappa)}}
		+ \frac{c_1(k)}{ \beta^{kb/2}}\right)^N,
	\end{equation}
	uniformly for all $k,\beta\ge m$. Because the left-hand
	side of \eqref{eq0} is bounded above by 
	$(\mathcal{T}_1+\mathcal{T}_2)^N+N\mathcal{T}_2$, it follows that
	\begin{align}
		&\P\left\{ \max_{1\le j\le N}|u_t(x^{(j)})| 
			<\theta(\log R)^{\nicefrac14}\right\}\\\notag
		&\hskip1.5in\le \left( 1 - \frac{a(\varkappa)^{-\frac12}}{R^{2(3\theta)^4 a(\varkappa)}}
			+ \frac{c_1(k)}{\beta^{kb/2}}\right)^N 
			+\frac{c_1(k)N}{\beta^{kb/2}},
			\label{eq1}
	\end{align}
	
	Now we choose the various parameters as follows:
	We choose $N := \lceil R^q\rceil^d$ and $\beta :=  R^{1-q}/\log R$,
	where $q\in(0\,,1)$ is fixed, and let $k \geq 2$ be the smallest integer
	so that $qd-\frac12 kb (1-q)<-2$ so that $N\beta^{-kb/2}\le R^{-2}$. 
	In a cube of side length $2 (1+\sqrt t)R$, there are at least $N$ points 
	separated by ``D-distance" $2n\beta (1+\sqrt t) $ where $n:=[\log\beta]+1$. 
	Also choose $\theta>0$ small enough 
	so that $(3\theta)^4 a(\varkappa) <q$.
	For these choices of parameters,
	an application of the Borel--Cantelli lemma [together with a monotonicity
	argument] implies that
	$\liminf_{R\to \infty} (\log R)^{-\nicefrac14}\,u_t^{*}(R) >0$ a.s.
	See \cite{CJK} for more details of this kind of argument in a similar setting.
\end{proof}

\subsection{The second part}

Similarly as in the proof of Theorem \ref{th:boundedness}, we will need a result on the modulus of continuity of $u$.

\begin{lemma}\label{mod_cont}
	If $\sup_{x\in\R}|\sigma(x)|<\infty$, then there exists
	a constant $C=C(t)\in(0\,,\infty)$ such that 
	\begin{equation}
		\E\left(\left\vert u_t(x)-u_t(x')\right\vert^{2k}\right) 
		\le \left(\frac{Ck}{\sqrt\varkappa}\right)^k\cdot \|x-x'\|^k,
	\end{equation}
	uniformly for all $x,x'\in\R^d$ that satisfy $\|x-x'\|\le 
	(t\varkappa)^{\nicefrac12}$.
\end{lemma}

\begin{proof}
	Let $S_0:=\sup_{z\in\R}|\sigma(z)|$. Because
	$|f(z)|\le f(0)$ for all $z\in\R^d$, 
	the optimal form of the Burkholder--Davis--Gundy inequality
	\textnormal{(BDG)} and \eqref{eq:cont_ic} imply that
	\begin{equation}
		\left\|  u_t(x)-u_t(x')\right\|_{2k}
		\le \text{const} \cdot \|x-x'\| + 2S_0\sqrt{2kf(0)\,Q_t(x-x')},
	\end{equation}
	where
	\begin{equation}
		Q_t(w) := \int_0^t\d s\left(\int_{\R^d}\d y\
		\left| p_{t-s}(y-w)-p_{t-s}(y)\right|\right)^2
		\quad\text{for  $w\in\R^d$}.
	\end{equation}
	Lemma \ref{p_bound} and a small computation implies readily that
	$Q_t(w) \le\textnormal{const}\cdot \|w\| \sqrt{ t /\varkappa}$ 
	whenever $\|w\|\le (t\varkappa)^{\nicefrac12}$; and the lemma
	follows  from these observations.
\end{proof}

\begin{lemma}\label{exp_bound} 
	Choose and fix $t>0$, and suppose that
	$\sigma$ is  bounded. Then there exists a constant $C\in (0\,,\infty)$ such that 
	\begin{equation}
		\E\left[\sup_{\substack{x,x'\in T: \\\|x-x'\|\le \delta}}
		\exp\left(\frac{\sqrt\varkappa |u_t(x)-u_t(x')|^2}{C\delta}\right)
		\right]\le \frac{2}{\delta},
	\end{equation}
	uniformly for every $\delta\in (0\,,(t\varkappa)^{\nicefrac12}]$ 
	and every cube $T\subset \R^d$ of side length at most $1$.
\end{lemma}

As the proof is quite similar to the proof of \cite[Lemma 6.2]{CJK},
we leave the verification to the reader. Instead we prove the following result,
which readily implies \eqref{eq:sigma:bdd2}, and thereby
completes our derivation of Theorem \ref{th:sigma:bdd}.

\begin{proposition} 
	If $\sigma$ is bounded uniformly away from zero
	and infinity, then $u_t^*(R)\asymp(\log R)^{\nicefrac12}$
	a.s.
\end{proposition}

\begin{proof}
	We may follow the proof of Proposition \ref{prop2},
	but use Lemma \ref{tail_bound} instead of Lemma \ref{tail_bound0},
	in order to establish that
	$\liminf_{R\to \infty}(\log R)^{-1/2}u_t^{*}(R)>0$ a.s.
	We skip the details, as they involve making only routine changes
	to the proof of Proposition \ref{prop2}.
	
	It remains to prove that
	\begin{equation}\label{ustar_upbd}
		u_t^{*}(R)  = O\left( (\log R)^{\nicefrac12}\right)
		\qquad(R\to\infty)\quad\text{a.s.\ for all $t>0$.}
	\end{equation}
	
	It suffices to consider the case that $R\gg t$.
	Let us divide the cube $[0\,,R]^d$ into subcubes 
	$\Gamma_1,\Gamma_2,\ldots$
	such that the $\Gamma_j$'s have common side length 
	$a := \text{const}\cdot (t\varkappa)^{\nicefrac12}$ 
	and the distance between any two points in $\Gamma_j$ 
	is at most $(t\varkappa)^{\nicefrac12}$.
	The total number $N$ of such subcubes is $O(R^d)$.
	
	We now apply Lemmas \ref{exp_bound} and \ref{tail_bound} as follows:
	\begin{align}
		&\P\left\{\sup_{x\in [0,R]^d} |u_t(x)|>2b(\ln R)^{%
			\nicefrac12}\right\} \\\notag
		& \le \P\left\{ \max_{1\le j\le N} |u_t(x_j)|>b (\ln R)^{%
			\nicefrac12}\right\} 
			+\P\left\{ \max_{1\le j\le N} \sup_{x,y\in \Gamma_j}
			|u_t(x)-u_t(y)|>b (\ln R)^{\nicefrac12} \right\} \\\notag
		&\le \textnormal{const}\cdot R^d\e^{-c_2b^2\ln R}+
			\frac{\textnormal{const}\cdot R^d}{(t\varkappa)^{%
			\nicefrac12}\exp\left(b^2\ln 
			(R)/Ct^{\nicefrac12}\right)}.
	\end{align}
	Consequently, 
	\begin{equation}
		\sum_{m= 1}^\infty
		\P\left\{\sup_{x\in [0,m]^d} |u_t(x)|>2b(\ln m)^{\nicefrac12}\right\}
		<\infty,
	\end{equation}
	provided that we choose $b$ sufficiently large.
	This, the Borel--Cantelli Lemma, and a  monotonicity argument 
	together complete the proof of \eqref{ustar_upbd}.
\end{proof}

\section{Proof of Theorem \ref{th:pa_f(0)Finite}}

Let us first establish some point estimates for the tail probability
of the solution $u$ to \textnormal{(PAM)}. Throughout this subsection the assumptions
of Theorem \ref{th:pa_f(0)Finite} are in force.

\begin{lemma}\label{lem:pa_ProbEst}
	For every $t>0$,
	\begin{equation}
		\limsup_{\lambda\to\infty}\sup_{x\in\R^d}
		\frac{\log\P \{|u_t(x)|\ge \lambda\}}{(\log \lambda)^2}
		\le -\frac{1}{4tf(0)}.
	\end{equation}
	Additionally, for every $t>0$,
	\begin{equation}
		\liminf_{\lambda\to\infty}\inf_{x\in\R^d}
		\frac{\log\P \{|u_t(x)|\ge \lambda\}}{(\log \lambda)^2}
		\ge -\frac{4tf(0)}{(\Lambda_d a_t)^2},
	\end{equation}
	where $\Lambda_d$ and $a_t=a_t(f\,,\varkappa)$ were defined in Proposition
	\ref{pr:pa_mombd}.
\end{lemma}

\begin{proof}
	Let $\log_+(z):=\log(z\vee\e)$ for all real numbers $z$. 
	Proposition \ref{pr:pa_mombd} and
	 Lemma 3.4 of the companion paper \cite{CJK} together imply that
	if $0<\gamma<(4tf(0))^{-1}$, then
	$\E\exp( \gamma|\log_+(u_t(x))|^2)$ is bounded
	uniformly in $x\in\R^d$. The
	first estimate of the lemma follows from this by an application
	of Chebyshev's inequality.
	
	As regards the second bound, we apply the Paley--Zygmund inequality 
	\eqref{eq:PZ} in conjunction with Proposition \ref{pr:pa_mombd} as follows:
	\begin{equation}
		\P\left\{ |u_t(x)|\ge \frac{1}{2}\|u_t(x)\|_{2k}\right\} 
		\ge  \frac{\left(\E\left(|u_t(x)|^{2k}\right)\right)^2}{4
		\E\left(|u_t(x)|^{4k}\right)}
		\ge \frac14 \e^{ k^2\left[ 8\Lambda_da_t-16tf(0)\right]}\cdot \left(\underline{u}_0/\overline{u}_0\right)^{4k}.
	\end{equation}
	Let us denote $\gamma=\gamma(\varkappa\,,t)
	:=16 tf(0)-8\Lambda_da_t>0$. A second application of Proposition
	\ref{pr:pa_mombd} then yields the following pointwise bound:
	\begin{equation}
		\P\left\{|u_t(x)|\ge \frac{\underline{u}_0}{2} \e^{2\Lambda_da_tk}\right\}
		\ge \frac{1}{4}\e^{-\gamma k^2}\left(\underline{u}_0/\overline{u}_0\right)^{4k}.
	\end{equation}
	The second assertion of the lemma follows from this
	and the trivial estimate $\gamma\le 16tf(0)$,
	because we can consider  $\lambda$ between
	$\frac{\underline{u}_0}{2}\exp(2\Lambda_d a_t k)$ and $\frac{\underline{u}_0}{2}\exp(2\Lambda_d a_t(k-1))$.
\end{proof}

Owing to the parameter dependencies pointed out in 
Proposition \ref{pr:pa_mombd}, Theorem \ref{th:pa_f(0)Finite}
is a direct consequence of the following result.

\begin{proposition}\label{pam}
	For the parabolic Anderson model, the following holds: For all $t>0$,
	there exists a constant $\theta_t\in(0\,,\infty)$---independent
	 of $\varkappa$---such that
	\begin{equation}
		\frac{\Lambda_d a_t}{(8tf(0))^{\nicefrac12}}
		\le \liminf_{R\to\infty} \frac{\log u_t^*(R)}{(\log R)^{\nicefrac12}}
		\le \limsup_{R\to\infty} \frac{\log u_t^*(R)}{(\log R)^{\nicefrac12}}
		\le \theta_t,
	\end{equation}
	where $\Lambda_d$ and $a_t=a_t(f\,,\varkappa)$ were defined in Proposition
	\ref{pr:pa_mombd}.
\end{proposition}

\begin{proof}
	Choose and fix two positive and finite numbers $a$ and $b$ that satisfy
	the following:
	\begin{equation}\label{eq:ab_0}
		a<\frac{1}{4tf(0)},
		\quad
		b>\frac{4tf(0)}{(\Lambda_d a_t)^2}.
	\end{equation}
	According to Lemma \ref{lem:pa_ProbEst}, the following holds for
	all $\lambda>0$ sufficiently large:
	\begin{equation}\label{eq:ab}
		\e^{-b(\log \lambda)^{2}} \le 
		\P\left\{ |u_t(x)|\ge \lambda \right\} \le 
		\e^{-a(\log \lambda)^{2}}.
	\end{equation}
	Our goal is twofold: First, we would like to prove that
	with probability one
	$\log  | u_t^*(R) |\asymp (\log R)^{\nicefrac12}$	as $R\to\infty$;
	and next to estimate the constants in ``$\asymp$.''
	
	We first derive an almost sure asymptotic 
	lower bound for $\log|u_t^*(R)|$. 
	
	Let us proceed as we did in our estimate of \eqref{eq0}.
	We introduce free parameters  $\beta,k,N\ge1$ [to be chosen later]
	together with $N$ points $x^{(1)},\ldots,x^{(N)}$.
	We will assume that $D(x^{(i)}\,,x^{(j)})\ge 2n\beta(1+\sqrt t)$ 
	where $D(x\,,y)$ was defined in \eqref{d} and 
	$n:=[\log\beta]+1$ as in Lemma \ref{lem:un-ind}. 
	If $\xi >0$ is an arbitrary parameter, then
	our localization estimate (Lemma \ref{lem:u-Ubl}) yields
	the following for all sufficiently-large values of $R$ [independently
	of $N$ and $\beta$]:
	\begin{align}\notag
		&\P\left\{\max_{1\le j\le N}|u_t(x^{(j)})|<
			\e^{\xi \sqrt{\log R}}\right\}\\\notag
		& \le \P\left\{ \max_{1\le j\le N}|U_t^{(\beta,n)}(x^{(j)})|<
			2\e^{\xi \sqrt{\log R}}\right\} 
			+\P\left\{ \max_{1\le j\le N}
			\left| u_t(x^{(j)})-U_t^{(\beta,n)}(x^{(j)})\right| 
			>\e^{\xi \sqrt{\log R}}\right\}\\
		&\le  \left( 1- \P\left\{ \left| U_t^{(\beta,n)}(x^{(1)})\right| 
			\ge 2\e^{\xi \sqrt{\log R}}\right\}\right)^N
			+ \frac{N G^kk^{k/2}\e^{Fk^2}}{\beta^{kb/2}\e^{k\xi 
			\sqrt{\log R}}}.
	\label{eq2}
	\end{align}
	And we estimate the remaining probability by similar means, viz.,
	\begin{align}\label{eq3}
		& \P\left\{\left| U_t^{(\beta,n)}(x^{(1)})\right| \ge 2
			\e^{\xi \sqrt{\log R}}\right\}\\\notag
		&\ge \P\left\{\left| u_t(x^{(1)})\right| \ge 
			3\e^{\xi \sqrt{\log R}}\right\} - \P\left\{
			\left| u_t(x^{(1)})-U_t^{(\beta,n)}(x^{(1)})\right| >
			\e^{\xi \sqrt{\log R}}\right\}\\\notag
		 &\ge \exp\left(
		 	-b\left\{\log\left(3\e^{\xi \sqrt{\log R}}\right)\right\}^{2}
			\right) -\frac{NG^kk^{k/2}\e^{Fk^2}}{\beta^{kb/2}
			\e^{k \xi \sqrt{\log R}}}.
	\end{align}
	We now fix our parameters $N$ and $\beta$ 
	as follows: First we choose an arbitrary $\theta\in(0\,,1)$,
	and then select $N:=\lceil R^{\theta}\rceil^d$ 
	and $\beta:= R^{1-\theta}/\log R$.
	For these choices, we can apply \eqref{eq3} in  \eqref{eq2} and deduce
	the bound
	\begin{align}
		&\P\left\{ \max_{1\le j\le N}|u_t(x_j)|< \e^{\xi \sqrt{%
			\log R}}\right\} \\\notag
		&\le \left( 1-\frac{\textnormal{const}}{R^{b\xi ^{2}}}+
			\frac{G^kk^{k/2}\e^{Fk^2}(\log R)^k}{%
			R^{(kb(1-\theta)-2\theta d)/2}\,
			\e^{\xi k\sqrt{\log R}}}\right)^N
			+\frac{G^kk^{k/2}\e^{Fk^2}(\log R)^k}{%
			R^{(kb(1-\theta)-2\theta d)/2}\,
			\e^{\xi k\sqrt{\log R}}}.
	\end{align}
	Now we choose our remaining parameters
	$k$ and $\xi $ so that $\frac12 kb(1-\theta)-\theta d>2$ and 
	$b\xi ^{2}<\theta/2$. In this way we obtain
	\begin{equation}
		\P\left\{\max_{1\le j\le N}|u_t(x_j)|<
		\e^{\xi \sqrt{\log R}}\right\} \le\exp\left(
		-CR^{\theta/2}\right) +\frac{C}{R^2}.
	\end{equation}
	In a cube of side length $2(1+\sqrt t)R$, there are at least $N$ 
	points separated by ``$D$-distance'' $2(1+\sqrt t) \beta n$. 
	Therefore, the Borel--Cantelli Lemma and a  
	monotonicity argument together imply that
	$\liminf_{R\to\infty} \exp\{-\xi (\log R)^{\nicefrac12}\}u_t^*(R)>1$ almost surely.
	We can first let $\theta\downarrow 1$, then $\xi \uparrow (2b)^{-\nicefrac12}$,
	and finally $b\uparrow 4tf(0)/(\Lambda_d a_t)^2$---in 
	 this order---in order to complete our derviation of the stated a.s.\ asymptotic
	lower bound for $u_t^*(R)$.

	For the other direction, we begin by applying \eqref{eq:cont_ic} and \textnormal{(BDG)}:
	\begin{align}
		&\left\| u_t(x)-u_t(y)\right\|_{2k} \leq \text{const} \cdot \|x-x'\| \\\notag
		&+ 2\left(4kf(0)\int_0^t\|u_s(0)\|_{2k}^2\,
			\d s\left[ \int_{\R^d}\d w \left| p_{t-s}(w-x)-p_{t-s}(w-y)\right|
			\right]^2\right)^{\nicefrac12}.
	\end{align}
	We apply Proposition \ref{pr:pa_mombd} to estimate $\|u_s(0)\|_{2k}$,
	and Lemma \ref{p_bound} to estimate the integral that involves
	the heat kernel. By arguments similar as in Lemma \ref{mod_cont}, we find that there exists 
	$C=C(t)\in(0\,,\infty)$---independently of $(x\,,y\,,k\,,\varkappa)$---such that
	uniformly for all $x,y\in\R^d$ with $\|x-y\|\le (t\varkappa)^{\nicefrac12}$,
	\begin{equation}
		\E\left(|u_t(x)-u_t(y)|^{2k}\right) \le (Ck)^k
		\e^{4tf(0)k^2} \frac{\|x-y\|^k}{\varkappa^{k/2}}.
	\end{equation}
	By arguments similar to the ones that led to (7.12) in the companion
	paper \cite{CJK} we can show that
	\begin{equation}\label{eq:M1}
		\E\left(\sup_{\substack{x,y\in T: \\
		\|x-y\| \le \sqrt{t\varkappa}}} 
		\left| u_t(x)-u_t(y)\right|^{2k}\right)
		\le  C_1^k \e^{C_2k^2}
	\end{equation}
	(where $C_1$ and $C_2$ depend only on $t$), uniformly over cubes $T$ with side lengths at most $1$. 
	[The preceding should be compared with the result of
	Lemma \ref{mod_cont}.] Now that we are armed with \eqref{eq:M1},
	we may proceed to complete the proof of the theorem
	as follows: We split $[0\,,R]^d$ into subcubes of side length 
	$a$ each of which is contained in a ball of radius $\frac12(t\varkappa)^{%
	\nicefrac12}$ centered 
	around its midpoint. Let $\mathcal{C}_R$ denotes the collection of all mentioned subcubes and $\mathcal{M}_R$ the set of their midpoints. For all $\zeta>0$, we have:
	\begin{align}
		\P\left\{ u_t^{*}(R) > 2\e^{\zeta\sqrt{\log R}}\right\}
		  &\le \P\left\{ \max_{x \in \mathcal{M}_R} |u_t(x)| >
			\e^{\zeta\sqrt{\log R}}\right\}\\\notag
		& \qquad + \P\left\{ \sup_{T\in\mathcal{C}_R} \sup_{x,y\in T}|u_t(x)-u_t(y)| >
			\e^{\zeta\sqrt{\log R}}\right\}, \\ \notag
			&\le O(R^d) \cdot\P\left\{ |u_t(0)| >
			\e^{\zeta\sqrt{\log R}}\right\}\\\notag
		& \qquad + \sum_{T\in\mathcal{C}_R}
			\P\left\{ \sup_{x,y\in T}|u_t(x)-u_t(y)| >
			\e^{\zeta\sqrt{\log R}}\right\}.
	\end{align}
	We use the notation set forth in \eqref{eq:ab}, together with
	\eqref{eq:M1}, and deduce the following estimate:
	\begin{equation}
		\P\left\{ u_t^{*}(R) > 2\e^{\zeta\sqrt{\log R}}\right\} 
		=O(R^d)
		\cdot\left[
		\e^{-a\zeta^2\log R}+
		\frac{C_1^k\e^{C_2k^2}}{\e^{2k\zeta \sqrt{\log R}}}\right],
	\label{pam:up}
	\end{equation}
	as $R\to\infty$.
	Now choose $k:=[(\log R)^{\nicefrac12}]$ and $\zeta$ 
	large so that the above is summable in $R$, as the variable 
	$R$ ranges over all positive integers. 
	The Borel-Cantelli Lemma and a standard 
	monotonicity argument together imply that with probability one,
	$\limsup_{R\to\infty} (\log R)^{-\nicefrac12}\log u_t^*(R)
	\le\zeta$.
	[Now $R$ is allowed to roam over all positive reals.]
	From the way in which $\zeta$ is chosen, it is clear that 
	$\zeta$ does not depend on $\varkappa$.
\end{proof}

\section{Riesz kernels}

Now we turn to the case where the correlation function is of the
Riesz form; more precisely, we have 
$f(x)=\textnormal{const}\cdot \|x\|^{-\alpha}$
for some $\alpha\in(0\,,d\wedge 2)$. 
We begin this discussion by establishing some
moment estimates for the solution $u$ to \textnormal{(PAM)}.
Before we being our analysis, let us recall some well-known
facts from harmonic analysis (see for example \cite{M}).

For all $b\in(0\,,d)$ define
$\mathcal{R}_b(x) := \|x\|^{-b}$ $(x\in\R^d).$
This is a rescaled Riesz kernel with index $b\in(0\,,d)$; it is a locally integrable
function whose Fourier transform is defined, for all $\xi\in\R^d$, as
\begin{equation}\label{eq:Riesz:FT}
	\hat{\mathcal{R}}_b(\xi) = C_{d,d-b} \mathcal{R}_{d-b}(\xi),
	\quad\text{where}\quad
	C_{d,p} := \frac{\pi^{d/2}2^{d-p}\Gamma((d-p)/2)}{\Gamma(p/2)}.
\end{equation}
We may note that the correlation function $f$ considered in this section is proportional to
$\mathcal{R}_\alpha$. We note also that the Fourier transform
of \eqref{eq:Riesz:FT} is understood in the sense of generalized
functions. Suppose next that $a,b\in(0\,,d)$ satisfy $a+b<d$, and note that
$\hat{\mathcal{R}}_{d-a}(\xi)\hat{\mathcal{R}}_{d-b}(\xi)
=\{C_{d,a}C_{d,b}/C_{d,a+b}\}\hat{\mathcal{R}}_{d-(a+b)}(\xi).$
In other words, whenever $a,b,a+b\in(0\,,d)$,
\begin{equation}\label{eq:Riesz:conv}
	\mathcal{R}_{d-a} * \mathcal{R}_{d-b}= 
	\frac{C_{d,a}C_{d,b}}{C_{d,a+b}}
	\mathcal{R}_{d-(a+b)},
\end{equation}
where the convolution is understood in the sense of generalized functions.

\subsection{Riesz-kernel estimates}

We now begin to develop several inequalities for the solution $u$ to \textnormal{(PAM)}
in the case that $f(x)=\textnormal{const}\cdot\|x\|^{-\alpha}=
\textnormal{const}\cdot \mathcal{R}_\alpha(x)$.

\begin{proposition}\label{pr:pa_mom_Riesz}
	There exists positive and finite constants 
	$\underline{c}=\underline{c}(\alpha\,,d)$ and
	$\bar{c}=\bar{c}(\alpha\,,d)$ such that
	\begin{equation}
		\underline{u}_0^k\exp\left(
		\underline{c}t \frac{k^{(4-\alpha)/(2-\alpha)}}{
		\varkappa^{\alpha/(2-\alpha)}}
		\right)\le \E\left(|u_t(x)|^k\right) \le \overline{u}_0^k \exp\left( \bar{c}t
		\frac{k^{(4-\alpha)/(2-\alpha)}}{\varkappa^{\alpha/(2-\alpha)}} \right),
	\end{equation}
	uniformly for all $x\in\R^d$, $t,\varkappa>0$, and $k\ge 2$, where $\underline{u}_0$ and $\overline{u}_0$ are defined in \eqref{ubar}.
\end{proposition}

\begin{remark}
	We are interested in what Proposition \ref{pr:pa_mom_Riesz}
	has to say in the regime in which $t$ is fixed, $\varkappa\approx 0$,
	and $k\approx\infty$. However, let us spend a few extra lines and 
	emphasize also the following somewhat different consequence
	of Proposition \ref{pr:pa_mom_Riesz}. Define for all $k\ge 2$,
	\begin{equation}\begin{split}
		\underline{\lambda}(k) &:= \liminf_{t\to\infty}
			\inf_{x\in\R^d} \frac1t \log\E\left(|u_t(x)|^k\right)\\
		\overline{\lambda}(k) &:= \limsup_{t\to\infty}
			\sup_{x\in\R^d} \frac1t \log\E\left(|u_t(x)|^k\right).
	\end{split}\end{equation}
	These are respectively the lower and upper uniform Lyapunov 
	$L^k(\P)$-exponents of the parabolic Anderson model driven by
	Riesz-type correlations. Convexity
	alone implies that if $\underline{\lambda}(k_0)>0$ for some
	$k_0>0$, and if $\overline{\lambda}(k)<\infty$ for all $k\ge k_0$,
	then $\underline{\lambda}(k)/k$ and $\overline{\lambda}(k)/k$
	are both strictly increasing for $k>k_0$. Proposition \ref{pr:pa_mom_Riesz}
	implies readily that the common of these increasing sequences is $\infty$.
	In fact, we have the following sharp growth rates, which appear to 
	have not been known previously:
	\begin{equation}
		\frac{\underline{c}}{\varkappa^{\alpha/(2-\alpha)}}\le\liminf_{k\to\infty}
		\frac{\underline{\lambda}(k)}{k^{2/(2-\alpha)}} \le
		\limsup_{k\to\infty}
		\frac{\overline{\lambda}(k)}{k^{2/(2-\alpha)}} \le 
		\frac{\overline{c}}{\varkappa^{\alpha/(2-\alpha)}}.
	\end{equation}
	These bounds can  be used to study further the large-time intermittent
	structure of the solution to the parabolic Anderson model driven by
	Riesz-type correlations. We will not delve into this matter here.
	\qed
\end{remark}

\begin{proof}
	Recall that $f(x)=A\cdot\|x\|^{-\alpha}$;
	we will, without incurring much loss of generality, that $A=1$.
	
	We first derive the lower bound on the moments of $u_t(x)$. 
	Let $\{b^{(j)}\}_{j=1}^k$ denote $k$ independent 
	standard Brownian motions
	in $\R^d$.
	We may apply Lemma \ref{lem:DC} to see that
	\begin{equation}
		\E\left( |u_t(x)|^k\right) \ge \underline{u}_0^k
		\E\left[\exp\left(
		\mathop{\sum\sum}\limits_{%
		1\le i\ne j \le k}
		\int_0^t \frac{ \varkappa^{-\alpha/2}\,\d s}{\left\| b_s^{(i)}-b_s^{(j)}
		\right\|^{\alpha}}\right)\right].
	\end{equation}
	
	We can use the preceding to obtain a large-deviations lower bound
	for the $k$th moment of $u_t(x)$ as follows: Note that 
	$\int_0^t \|b_s^{(i)}-b_s^{(j)}\|^{-\alpha}\,\d s
	\ge (2\epsilon)^{-\alpha}t
	\1_{\Omega_\epsilon}$ a.s.,
	where $\Omega_\epsilon$ is defined as the  event
	$\{\max_{1\le l\le k}\sup_{s\in[0,t]}\|b_s^{(l)}\|\le\epsilon\}.$
	Therefore,
	\begin{equation}\label{lomo:r}
		\E\left( |u_t(x)|^k\right) \ge\underline{u}_0^k \sup_{\epsilon>0}\left[
		\exp\left(\frac{k(k-1)t}{(2\epsilon
		\sqrt{\varkappa})^{\alpha}}\right)
		\cdot \P(\Omega_\epsilon)\right].
	\end{equation}
	Because of an eigenfunction expansion 
	\cite[Theorem 7.2, p.\  126]{PortStone}
	there exist constants
	 $\lambda_1=\lambda_1(d)\in(0\,,\infty)$, and $c=c(d)\in(0\,,\infty)$
	such that
	\begin{equation}\begin{split}
		\P(\Omega_\epsilon) &=\left( \P\left\{\sup_{s\in[0,t/\epsilon^2]}
		\|b_s^{(1)}\|\le 1 \right\} \right)^k \ge c^k\e^{ -
		kt\lambda_1/\epsilon^2},
	\end{split}\end{equation}
	 uniformly for all  $k\ge 2$ and $\epsilon\in(0\,, t^{\nicefrac12}]$.
	 And, in fact, $\lambda_1$ is the smallest
	positive  eigenvalue of the Dirichlet Laplacian on the unit
	ball of $\R^d$. Thus,
	\begin{equation}
		\E\left( |u_t(x)|^k\right) \ge (c\underline{u}_0)^k \sup_{\epsilon\in (0,t^{\nicefrac12}]}
		\left[\exp\left(\frac{k(k-1)t}{(2\epsilon
		\sqrt{\varkappa})^{\alpha}}-\frac{kt\lambda_1}{\epsilon^2}\right)\right].
	\end{equation}
	The supremum of the expression inside the exponential is at least
	$\text{const}\cdot
	tk\cdot k^{2/(2-\alpha)}/\varkappa^{\alpha/(2-\alpha)},$
	where ``const'' depends only on $(\alpha\,,d)$.
	This proves the asserted lower bound on the $L^k(\P)$-norm of $u_t(x)$.
	
	We adopt a different route for the upper bound.
	Let $\{\bar{R}_\lambda\}_{\lambda>0}$ denote the
	resolvent corresponding to $\sqrt{2}$ times a Brownian motion
	in $\R^d$ with diffusion coefficient $\varkappa$. In other
	words, $\bar{R}_\lambda f := \int_0^\infty\exp(-\lambda s)
	(p_{2s}*f)\,\d s=(\nicefrac12)(R_{\lambda/2}f)$.
	Next define
	\begin{equation}
		Q(k\,,\beta) := z_k\sqrt{(\bar{R}_{2\beta/k}f)(0)}
		\qquad\text{for all $\beta>0$ and $k\ge 2$,}
	\end{equation}
	where $z_k$ is the optimal constant, due to Davis
	\cite{Davis}, in the Burkholder--Davis--Gundy
	inequality for the $L^k(\P)$ norm of continuous martingales
	\cite{Burkholder,BDG,BG}.
	We can combine \cite[Theorem 1.2]{FK:TRAMS} and 
	Dalang's theorem \cite{Dalang:99}
	to conclude that, because the solution to \textnormal{(PAM)} exists,
	$(\bar{R}_ \lambda f)(0)<\infty$ for all $\lambda>0$.
	The proof of \cite[Theorem 1.3]{FK:TRAMS}
	and Eq.\ (5.35) therein [\emph{loc.\ cit.}] together  imply that
	if $Q(k\,,\beta)<1$ then
	$\e^{-\beta t/k}\|u_t(x)\|_k\le \overline{u}_0/(1-Q(k\,,\beta))$
	uniformly for all $t>0$ and $x\in\R^d$.
	In particular, if $Q(k\,,\beta)\le \tfrac12$, then 
	\begin{equation}\label{eq:goodgoal}
		\E \left(|u_t(x)|^k \right) \le \e^{\beta t}2^k\overline{u}_0^k.
	\end{equation}
	According to  Carlen and Kree  \cite{CK},
	$z_k\le 2\sqrt{k}$; this is the inequality that led also to
	\textnormal{(BDG)}. Therefore, \eqref{eq:goodgoal} holds
	as soon as $k(\bar{R}_{2\beta/k}f)(0)<\nicefrac{1}{16}$. Because
	both Brownian motion and $f$ satisfy scaling relations, a simple
	change of variables shows that
	$(\bar{R}_\lambda f)(0) = c_2
	\lambda^{-(2-\alpha)/2} \varkappa^{-\alpha/2},$
	where $c_2$ is also a nontrivial constant 
	that depends only on $(d\,,\alpha)$.
	Therefore, the condition $k(\bar{R}_{2\beta/k}f)(0)<\nicefrac{1}{16}$---shown
	earlier to be sufficient for \eqref{eq:goodgoal}---is equivalent to
	the assertion that
	$\beta> k \cdot c_3 k^{2/(2-\alpha)}/\varkappa^{\alpha/(2-\alpha)}$
	for a nontrival constant $c_3$ that depends only on
	$(d\,,\alpha)$. Now we choose
	$\beta:=2k \cdot c_3 k^{2/(2-\alpha)}/\varkappa^{\alpha/(2-\alpha)}$,
	plug this choice in \eqref{eq:goodgoal}, and deduce the upper bound.
\end{proof}

Before we proceed further, let us observe that, 
in accord with \eqref{eq:Riesz:conv},
\begin{equation}
	f(x)=\frac{\text{const}}{\|x\|^{\alpha}}  =(h*h)(x) = (h*\tilde{h})(x)
	\quad\text{with}\quad
	h(x) :=\frac{\text{const}}{\|x\|^{(d+\alpha)/2}},
\end{equation}
where  the convolution is understood in the sense of generalized functions.

 As in \eqref{rho}, we can define 
 $h_n(x):=h(x)\hat{\varrho}_n(x)$ and $f_n=(h-h_n)*(\tilde{h}-\tilde{h}_n)$.  
  
\begin{lemma}\label{lem:new}
	For all $\eta\in(0\,,1\wedge\alpha)$
	there exists a constant $A:=A(d\,,\varkappa\,,\alpha\,,\eta)\in(0\,,\infty)$
	such that 
	$( p_s* f_n)(0) \le   A n^{-\eta}\cdot s^{-(\alpha-\eta)/2}$
	for all $s>0$ and $n\ge 1$.
\end{lemma}

\begin{proof}
	Because $f_n\le h*(h-h_n)$, it follows that
	\begin{align}\notag
		(p_s*f_n)(0) &\le \left[ (p_s*h*h)(0) - (p_s*h*h_n)(0)\right]\\
		&= \int_{\R^d}\d y\int_{\R^d}\d z\
			p_s(z) h(y)h(y-z) \left( 1-\hat{\varrho}_n(y-z) \right)\\\notag
		&\le \textnormal{const}\cdot
			\int_{\R^d}\frac{\d y}{\|y\|^{(d+\alpha)/2}}\int_{\R^d}
			\frac{p_s(z) \,\d z}{\|y-z\|^{(d+\alpha)/2}}\
			\left( 1\wedge \frac{\|y-z\|}{n}\right);
	\end{align}
	see \eqref{eqrho}. 
	
	Choose and fix some $\eta\in(0\,,1\wedge \alpha)$. Since 
	$1\wedge r\le r^\eta$ for all $r>0$,
	\begin{align}\notag
		(p_s*f_n)(0) &\le \frac{\textnormal{const}}{n^\eta}\cdot
			\int_{\R^d}\frac{\d y}{\|y\|^{(d+\alpha)/2}}\int_{\R^d}
			\frac{p_s(z)\,\d z}{\|y-z\|^{(d+\alpha-2\eta)/2}}\\\notag
		&=\frac{\textnormal{const}}{n^\eta}\cdot
			\int_{\R^d}\d y\int_{\R^d}\d z\
			\mathcal{R}_{(d+\alpha)/2}(y)p_s(z) 
			\mathcal{R}_{(d+\alpha-2\eta)/2}(z-y)\\
		&=\frac{\textnormal{const}}{n^\eta}\cdot
			\int_{\R^d} \|z\|^{-\alpha+\eta} p_s(z)\,\d z,
	\end{align}
	by \eqref{eq:Riesz:conv}, because $p_s$ is a rapidly-decreasing test function for all $s>0$.
	 A change of variable in the integral above proves the result.
\end{proof}

\begin{proposition}\label{pr:Riesz:h-h_n}
	For  every $\eta\in(0\,,1\wedge\alpha)$,
	the following holds uniformly for every $k\ge 2$, $\delta>0$, and  
	all predictable random fields $Z$:
	\begin{equation}
		\mathcal{M}_\delta^{(k)}\left(p*ZF^{(h)}
		-p*ZF^{(h_n)}\right)  \le\textnormal{const}
		\cdot \sqrt{\frac{ k}{n^\eta \cdot \delta^{(2-\alpha+\eta)/2}}}\
		\mathcal{M}_\delta^{(k)}(Z),
	\end{equation}
	where the implied constant depends only on 
	$(d\,,\varkappa\,,\alpha\,,\eta)$.
\end{proposition}

\begin{remark}
Proposition \ref{pr:Riesz:h-h_n} compares to Proposition \ref{pr:mk}. \qed
\end{remark}

\begin{proof}
	For notational simplicity, let us write
	\begin{equation}
		\Xi :=\left\| \left( p*ZF^{(h)} \right)_t(x)-
		\left( p*ZF^{(h_n)} \right)_t(x) \right\|_k.
	\end{equation}
	We apply first \textnormal{(BDG)}, 
	and then Minkowski's inequality, to see that for all $\delta>0$,
	\begin{align}\notag
		\Xi^2 &\le 4\left[\mathcal{M}_\delta^{(k)}(Z)\right]^2 k
			\int_0^t  \d s
			\int_{\R^d}\d y\ \int_{\R^d}\d z\ p_{t-s}(y) f_n(y-z) p_{t-s}(z) \\
			\notag
		&\le 4\e^{2\delta t}\left[\mathcal{M}_\delta^{(k)}(Z)\right]^2 k
			\int_0^\infty \e^{-2\delta r}\left( p_{2r}* f_n\right)(0)\,\d r\\
		&=2\e^{2\delta t}\left[\mathcal{M}_\delta^{(k)}(Z)\right]^2 k
			\int_0^\infty \e^{-\delta s}\left( p_s* f_n\right)(0)\,\d s.
			\label{ph-phn}
	\end{align}
	The appeal to Fubini's theorem is justified since:
	(i)  $p_r$ is a rapidly decreasing test function for all $r>0$;
	(ii) $p_r*p_r=p_{2r}$ by the Chapman--Kolmogorov equation;
	and  (iii) $p_r,f_n\ge 0$ pointwise for every $r>0$ and $n\ge 1$.
	Now we apply Lemma \ref{lem:new} in order to find that
	for all $\eta\in(0\,,1\wedge\alpha)$,
	\begin{equation}\begin{split}
		\Xi^2 &\le \textnormal{const}\cdot
			\frac{\e^{2\delta t}k}{n^{\eta}}
			\left[ \mathcal{M}_\delta^{(k)}(Z)
			\right]^2\int_0^\infty 
			\e^{-\delta s} s^{-(\alpha-\eta)/2}\,\d s\\
		&=\textnormal{const}\cdot
			\frac{\e^{2\delta t}k}{n^{\eta}}
			\left[ \mathcal{M}_\delta^{(k)}(Z)
			\right]^2 \delta^{-(2-\alpha+\eta)/2}.
	\end{split}\end{equation}
	Since the right-most term is independent of $x$, we can divide both sides
	by $\exp(2\delta t)$, optimize over $t$, 
	and then take square root to complete the proof.
\end{proof}

\subsection{Localization for Riesz kernels}

The next step in our analysis of Riesz-type correlations is to establish
localization; namely results that are similar to those of Section \ref{local}
but which are applicable to the setting of Riesz kernels.

\subsection{The general case}
Recall the random fields $U^{(\beta)}$, $V^{(\beta)}$,
and $Y^{(\beta)}$ respectively from \eqref{eq:U_beta},
\eqref{eq:V_beta}, and \eqref{eq:Y_beta}.
We begin by studying the nonlinear problem \textnormal{(PAM)}
in the presence of noise whose spatial correlation is determined
by $f(x)=\textnormal{const}\cdot\|x\|^{-\alpha}$.

\begin{proposition}\label{pr:Riesz:u-U}
	Let $u$ denote the solution to \textnormal{(PAM)}. For every $T>0$ and $\eta\in(0\,,1\wedge\alpha)$ 
	there exist finite and positive
	constants $\ell_i:=\ell_i(d\,,\alpha\,,T,\varkappa\,,\eta)$ 
	$[i=1,2]$, such that uniformly for $\beta>0$ and $k\ge 2$,
	\begin{equation}
		\sup_{t\in[0,T]} \sup_{x\in\R^d} 
		\E\left(\left |u_t(x)-U_t^{(\beta)}(x)\right|^k\right)
		\le\left(\frac{\ell_2 k}{\beta^\eta} \right)^{k/2}
		 \e^{\ell_1k^{(4-\alpha)/(2-\alpha)}}.
	\end{equation}
\end{proposition}

\begin{proof}
	Notice that
	\begin{equation}\begin{split}
		V^{(\beta)}_t(x) &= (p_t*u_0)(x)+ \left( p*U^{(\beta)} F^{(h)}\right)_t(x),\\
		Y^{(\beta)}_t(x) &= (p_t*u_0)(x)+ \left( p*U^{(\beta)} F^{(h_\beta)}\right)_t(x).
	\end{split}\end{equation}
	Proposition \ref{pr:Riesz:h-h_n} tells us that for all
	$\eta\in(0\,,1\wedge\alpha)$,
	\begin{equation}
		\mathcal{M}_\delta^{(k)}\left( V^{(\beta)}-
		Y^{(\beta)}\right) \le C_1\cdot
		\sqrt{\frac{k}{ \beta^\eta\cdot
		\delta^{(2-\alpha+\eta)/2}}}\, \mathcal{M}_\delta^{(k)}(U^{(\beta)}),
	\end{equation}
	where $C_1$ is a positive and finite constant that depends only on
	$(d\,,\varkappa\,,\alpha\,,\eta)$. 
	It follows from the definition \eqref{eq:M} that
	\begin{equation}
	\mathcal{M}_\delta^{(k)}\left( V^{(\beta)}-
		Y^{(\beta)}\right)\le\textnormal{const}\cdot
		\sqrt{\frac{k\varkappa^{(2-\alpha)/2}}{ \beta^\eta
		\cdot\delta^{(2-\alpha+\eta)/2}}}
	  \mathcal{M}_\delta^{(k)}(U^{(\beta)}),
	\end{equation}
	where ``const'' depends only on $(d\,,\varkappa\,,\alpha\,,\eta)$.
	In order to estimate the latter $\mathcal{M}_\delta^{(k)}$-norm we mimic
	the proof of the first inequality in Proposition \ref{pr:pa_mom_Riesz}
	to see that, for the same constant $\overline{c}$ as in the latter proposition,
	$\log \| U^{(\beta)}_t(x) \|_k \le
	\overline{u}_0+\overline{c}t k^{2/(2-\alpha)}/\varkappa^{\alpha/(2-\alpha)}$
	uniformly for all $x\in\R^d$, $t,\varkappa,\beta>0$, and $k\ge 2$.
	We omit the lengthy details because they involve making only small
	changes to the proof of the second inequality in Proposition
	\ref{pr:pa_mom_Riesz}. The end result is that
	\begin{equation}\label{cond:beta:lb0}
		\mathcal{M}_\delta^{(k)} (U^{(\beta)}) \le
		\sup_{t>0} \left[\overline{u}_0\exp\left\{ -\delta t + \overline{c}t
		\frac{k^{2/(2-\alpha)}}{\varkappa^{\alpha/(2-\alpha)}}\right\}\right]
		= \overline{u}_0,
	\end{equation}
	provided that 
	\begin{equation}\label{cond:beta:lb}
		\delta  > \overline{c} k^{2/(2-\alpha)}/\varkappa ^{\alpha/(2-\alpha)}.
	\end{equation}
	Therefore, the following is valid
	whenever $\delta$ satisfies \eqref{cond:beta:lb}:
	\begin{equation}\label{cond:beta:lb1}
		\mathcal{M}_\delta^{(k)}\left( V^{(\beta)}-Y^{(\beta)} \right) 
		\le C_1\cdot
		\sqrt{\frac{k}{ \beta^\eta\cdot\delta^{(2-\alpha+\eta)/2}}},
	\end{equation}
	where $C_1$ depends only on $(d\,,\varkappa\,,\alpha\,,\eta\,,\sigma(0)
	\,,\lip,\, \overline{u}_0)$.

	In order to bound $\|Y^{(\beta)}_t(x)-U^{(\beta)}_t(x)\|_k$,
	we apply \textnormal{(BDG)} and deduce that
	\begin{align}
		&\E\left( |Y_t^{(\beta)}(x)-U_t^{(\beta)}(x)|^k\right)\\\notag
		&\le \E\left( \left| 4k \int_0^t\d s
			\int_{[x-\beta\sqrt t,x+\beta\sqrt t]^c}\d y
			\int_{[x-\beta\sqrt t,x+\beta \sqrt t]^c}\d z\  
			h^{(*2)}_\beta(z-y) \mathcal{W}
			\right|^{k/2} \right)\\\notag
		&\le \left(  4k\int_0^t\d s
			\int_{[x-\beta\sqrt t,\,x+\beta\sqrt t]^c}\d y
			\int_{[x-\beta\sqrt t, \, x+\beta\sqrt t]^c}\d z\  f(z-y) 
			\| \mathcal{W} \|_{k/2}
			\right)^{k/2},
	\end{align}
	where we have used Minkowski's inequality in the last bound.
	Here, $h^{(*2)}_\beta := h_\beta * \tilde{h}_\beta$, and
	$\mathcal{W} := p_{t-s}(y-x)
			p_{t-s}(z-x)| U_s^{(\beta)}(y)|
			\cdot| U_s^{(\beta)}(z)|$.
	In particular,
	\begin{equation}
		\| \mathcal{W}\|_{k/2} \le
			p_{t-s}(y-x)p_{t-s}(z-x)\sup_{y\in \R^d}
			\left\| U^{(\beta)}_s(y) \right\|_k ^2,
	\end{equation}
	thanks to the Cauchy--Schwarz inequality. By the definition
	\eqref{eq:M} of $\mathcal{M}_\delta^{(k)}$,
	\begin{equation}
		\sup_{w\in\R^d}\left\| U^{(\beta)}_s(w) \right\|_k \le
		\e^{\delta s} \mathcal{M}_\delta^{(k)}\left( U^{(\beta)}\right)
		\qquad\text{for all $s>0$}.
	\end{equation}
	Therefore,
	\begin{equation}
		\| \mathcal{W}\|_{k/2} 
		\le \textnormal{const}\cdot
		\e^{2\delta s}p_{t-s}(y-x) p_{t-s}(z-x)
		\mathcal{M}_\delta^{(k)} (U^{(\beta)} )^2.
	\end{equation}
	
	Let us define 
	\begin{equation}\label{def:Theta}
		\Theta := \int_0^t\d s\mathop{\iint}_{\mathcal{A}\times
		\mathcal{A}}\d y\,\d z\ \  f(z-y)
		\e^{2\delta s}p_{t-s}(y-x) p_{t-s}(z-x),
	\end{equation}
	where we have written $\mathcal{A}:=[x-\beta\sqrt t\,,x+\beta\sqrt t]^c$,
	for the sake of typographical ease.
	Our discussion so far implies that
	\begin{equation}\label{eq:Theta:est}
		\left\| Y^{(\beta)}_t(x)-U^{(\beta)}_t(x)\right\|_k
		\le \textnormal{const}\cdot\sqrt{k}\
	 \mathcal{M}_\delta^{(k)} (U^{(\beta)} )
	\cdot\Theta^{\nicefrac12}.
	\end{equation}
	 We may estimate $\Theta$ as follows:
	\begin{equation}\begin{split}
		\Theta &\le \int_0^t
			\sup_{w\in\R^d}(p_{t-s}*f)(w)
			\e^{2\delta s}\, \d s\int_{[x-\beta\sqrt t,x+\beta\sqrt t]^c}
			p_{t-s}(y-x)\,\d y\\
		&\hskip.3in\le \textnormal{const}\cdot\int_0^t
			\sup_{w\in\R^d}(p_{t-s}*f)(w)
			\exp\left( -\frac{d \beta^2 t}{4\varkappa(t-s)}+2\delta s\right) \, \d s,
	\end{split}\end{equation}
	where we used \eqref{eqp} and ``const'' depends only on $(d\,,\alpha)$. 
	Because $p_{t-s}*f$ is a continuous positive-definite function,
	it is maximized at the origin. Thus, by scaling,
	\begin{equation}\label{eq:p*f}
		\sup_{w\in\R^d}\left( p_{t-s}*f\right)(w)  \le
		\frac{\textnormal{const}}{(t-s)^{\alpha/2}\varkappa^{\alpha/2}},
	\end{equation}
	where ``const'' depends only on $(d\,,\alpha)$. Consequently,
	\begin{align}
		\Theta &\le  \frac{\textnormal{const}}{\varkappa^{\alpha/2}}
			\cdot \int_0^t
			\exp\left( -\frac{d \beta^2 t}{4\varkappa(t-s)}+2\delta s\right) 
			\,\frac{\d s}{(t-s)^{\alpha/2}}\\\notag
		&\le\frac{\textnormal{const}}{\varkappa^{\alpha/2}}
			\cdot \e^{2\delta t} \,t^{(2-\alpha)/2}\int_0^1
			\e^{- d\beta^2 /(4\varkappa s)}
			\,\frac{\d s}{s^{\alpha/2}} 
			\le \frac{\textnormal{const}}{\varkappa^{\alpha/2}}
			t^{\nicefrac{(2-\alpha)}{2}}
			\exp\left(2\delta t-\frac{d \beta^2}{4\varkappa}\right).
	\end{align}
	It follows from the preceding discussion and \eqref{eq:Theta:est} that
	\begin{equation}\label{cond:beta:lb2}
		\mathcal{M}^{(k)}_\delta\left( Y^{(\beta)}-U^{(\beta)}\right)
		\le \frac{\textnormal{const}}{\varkappa^{\alpha/4}}
		\cdot \e^{-d\beta^2/(8\varkappa)}\, \sqrt{k},
	\end{equation}
	provided that $\delta$ satisfies  \eqref{cond:beta:lb}.
	
	Next we note that
	\begin{equation}\begin{split} 
		& \|u_t(x)-V_t^{(\beta)}(x)\|_k\\
		&\le \left\|
			\int_{(0,t)\times \R^d} p_{t-s}(y-x)\left[u_s(y)-U_s^{(\beta)}(y)\right]
			F^{(h)}(\d s\,\d y)\right\|_k \\
		&\le  \textnormal{const}\cdot\sqrt{k
			\int_0^t\d s\int_{\R^d}\d y\int_{\R^d}\d z\
			f(y-z) \widetilde{\mathcal{T}}},
	\end{split}\end{equation}
	where
	\begin{equation}\begin{split}
		\widetilde{\mathcal{T}} &:= 
		p_{t-s}(y-x)p_{t-s}(z-x)\left\| u_s(z)-U_s^{(\beta)}(z)\right\|_k
		\cdot\left\| u_s(y)-U^{(\beta)}_s(y) \right\|_k\\
		&\leq p_{t-s}(y-x)p_{t-s}(z-x)\sup_{y\in \R^d}\left\| u_s(y)-U_s^{(\beta)}(y)\right\|^2_k.
	\end{split}\end{equation}
	We then obtain 
	\begin{equation}\begin{split}
		\|u_t(x)-V_t^{(\beta)}(x)\|_k
		\le  \textnormal{const}\cdot \left(k\int_0^t
		\frac{\sup_{y}\|u_s(y)-U_s^{(\beta)}(y)\|_k^2}{\left((t-s)
		\varkappa\right)^{\alpha/2}}\,\d s\right)^{\nicefrac12},
	\end{split}\end{equation}
	 from similar calculations as before; see the derivation 
	 of \eqref{eq:p*f}. Consequently,
	 \begin{equation}\begin{split}\label{eq:u-Vb}
		 \mathcal{M}^{(k)}_\delta\left(u-V^{(\beta)}\right) 
		 	&\le \textnormal{const} \cdot k^{\nicefrac12} \mathcal{M}^{(k)}_\delta
				\left(u-U^{(\beta)}\right) \left(\int_0^\infty  \frac{\e^{-2\delta r}}{%
				\left(\varkappa r\right)^{\alpha/2}}\,\d r\right)^{\nicefrac12}\\
			&= \frac{\textnormal{const}\cdot k^{\nicefrac12}}{%
				\varkappa^{\alpha/4}\delta^{(2-\alpha)/4}}\,
				\mathcal{M}^{(k)}_\delta\left(u-U^{(\beta)}\right).
	 \end{split}\end{equation}
	Next we apply the decomposition \eqref{u:decompose} 
	and the bounds in \eqref{eq:u-Vb}, \eqref{cond:beta:lb2}, and  
	\eqref{cond:beta:lb1} 
	to see that
	\begin{align}
		&\mathcal{M}^{(k)}_\delta\left( u- U^{(\beta)}\right)\\\notag
		&\le \mathcal{M}^{(k)}_\delta\left(u - V^{(\beta)}\right)
			+ \mathcal{M}_\delta^{(k)}\left( V^{(\beta)}-Y^{(\beta)}\right)
			+ \mathcal{M}^{(k)}_\delta\left( Y^{(\beta)}-U^{(\beta)}\right) \\\notag
		&\le \frac{\textnormal{const}\cdot k^{\nicefrac12}}{%
			\varkappa^{\alpha/4}\delta^{(2-\alpha)/4}}
			\mathcal{M}^{(k)}_\delta\left(u-U^{(\beta)}\right) +C_1\cdot
			\sqrt{\frac{k}{\beta^\eta\cdot \delta^{(2-\alpha+\eta)/2}}}
			+ \frac{\textnormal{const}\cdot k^{\nicefrac12}}{\varkappa^{\alpha/4}}
			 \e^{-d\beta^2/(8\varkappa)}.
	\end{align}
	We now choose $\delta := Ck^{2/(2-\alpha)}/\varkappa^{\alpha/(2-\alpha)}$ 
	with $C>\overline{c}$ so large that the coefficient of
	$\mathcal{M}^{(k)}_\delta (u-U^{(\beta)} )$ in the preceding bound,
	is smaller that $\nicefrac12$. Because $\delta$ has a lower bound
	that holds uniformly for all $k\ge 1$, the preceding implies that
	\begin{equation}
		\mathcal{M}^{(k)}_\delta\left( u- U^{(\beta)}\right)
		\le \textnormal{const}\cdot\sqrt k\,\left[
		\beta^{-\eta/2} +  \e^{-d\beta^2/(8\varkappa)} \right] \le
		\textnormal{const}\cdot\sqrt k\, \beta^{-\eta/2},
	\end{equation}
	which has the desired result.
\end{proof}	

Recall the $n$th level Picard-iteration approximation $U^{(\beta,n)}_t(x)$ 
of $U_t^{(\beta)}(x)$ 
defined in \eqref{ubn}. The next two lemmas are the Picard-iteration 
analogues of Lemmas \ref{lem:u-Ubl} and \ref{lem:un-ind}.

\begin{lemma}\label{reisz:u-Ubl} 
	For every $T>0$ and $\eta\in(0\,,1\wedge\alpha)$ 
	there exist finite and positive
	constants $\ell_i:=\ell_i(d\,,\alpha\,,T,\varkappa\,,\eta\,,\sigma)$ 
	$[i=1,2]$, such that uniformly for $\beta>0$ and $k\ge 2$
	\begin{equation}
		\sup_{t\in[0,T]} \sup_{x\in\R^d} 
		\E\left(\left |u_t(x)-U_t^{(\beta,\, [\log \beta]+1)}(x)\right |^k\right)
		\le \left(\frac{\ell_2 k}{\beta^\eta} \right)^{k/2}
		\e^{\ell_1k^{(4-\alpha)/(2-\alpha)}}.
	\end{equation}
\end{lemma}

\begin{lemma}\label{riesz:un-ind}
	Choose and fix $\beta\ge 1,\, t>0$ and $n\ge 1$. 
	Also fix $x^{(1)},x^{(2)},\cdots \in \mathbf{R}^d$ such that 
	$D(x^{(i)}\,,x^{(j)})\ge 2 n\beta (1+\sqrt{t})$. Then 
	$\{ U_t^{(\beta, n)}(x^{(j)})\}_{j\in \Z}$ are independent 
	random variables. 
\end{lemma}	

We will skip the proofs, as they are entirely similar to the respective proofs
of Lemmas \ref{lem:u-Ubl} and \ref{lem:un-ind}, but apply the method of
proof of Proposition \ref{pr:Riesz:u-U} in place of  Lemma \ref{lem:localization}.

\subsection{Proof of Theorem \ref{th:pa_Riesz}}

The proof of this theorem is similar to that of 
Theorem \ref{th:pa_f(0)Finite}. Thanks to Proposition \ref{pr:pa_mom_Riesz}
and  \cite[Lemma 3.4]{CJK}, we have the following: There
exist positive and finite constants $a<b$, independently of
$\varkappa>0$, such that for all $x\in\R^d$ and $\lambda>\e$,
\begin{equation}
	a\e^{-b(\log \lambda)^{(4-\alpha)/2}\varkappa^{\alpha/2}}
	\le \P\left\{|u_t(x)|\ge \lambda \right\}
	\le b\e^{-a(\log \lambda)^{(4-\alpha)/2}\varkappa^{\alpha/2}}.
\end{equation}
	
Define, for the sake of typographical ease, 
\begin{equation}
	\mathcal{E}_M := \mathcal{E}_{M,\varkappa}(R)
	:= \exp\left(\frac{M\cdot (\log R)^{2/(4-\alpha)}}{%
	\varkappa^{\alpha/(4-\alpha)}}\right)
	\qquad \text{for all $M>0$}.
\end{equation}
For the lower bound, we  again choose $N$ points 
$x^{(1)},\ldots,x^{(N)}$ such that 
$D(x^{(i)}\,,x^{(j)})\ge 2n\beta(1+\sqrt t)$ whenever $i\neq j$;
see \eqref{d} for the definition of $D(x\,,y)$. 
Let $n :=[\log \beta]+1$ and choose and fix $\eta\in(0\,,1\wedge\alpha)$.
We apply Proposition  \ref{pr:Riesz:u-U} 
and the independence
of the $U^{(\beta,n)}(x^{(j)})$'s (Lemma  \ref{riesz:un-ind}) to see that
\begin{align}\notag
	&\P\left\{ \max_{1\le j\le N}|u_t(x^{(j)})|< 
		\mathcal{E}_M\right\}\\
	&\le  \left( 1- \P\left\{  | U_t^{(\beta,n)}(x^{(1)})|\ge 2 
		\mathcal{E}_M\right\} \right)^N
		+ \textnormal{const}\cdot
		\frac{N}{\beta^{k\eta/2}\mathcal{E}_M^k}\\
		\notag
	&\le \left( 1- \left[\P \left\{ | u_t(x^{(1)})|\ge 3\mathcal{E}_M\right\}
		-  \textnormal{const}\cdot
		\frac{N}{\beta^{k\eta/2}}\right] \right)^N
		+ \textnormal{const}\cdot
		\frac{N}{\beta^{k\eta/2}},
\end{align}
since $\mathcal{E}_M$ is large for $R$ sufficently large. Notice that the implied constants depend on 
$(\varkappa\,, t\,, k\,, d, \alpha\,,\eta\,,\sigma)$. 
Now we choose the various parameters involved [and in this order]:
Choose and fix some $\nu\in(0\,,1)$,
and then set $N:=\lceil R^{\nu}\rceil^d$
and $\beta:=R^{1-\nu}$. The following is valid for all $M>0$
sufficiently small, every $k$ sufficiently large,
and for the mentioned choices of $N$ and $\beta$:
\begin{equation}
	 \P\left\{ \max_{1\le j\le N}|u_t(x^{(j)})|<
	 \mathcal{E}_M\right\}
	 \le \textnormal{const}\cdot R^{-2}.
\end{equation}
Borel-Cantelli Lemma and a simple monotonicity argument 
together yield the bound,
\begin{equation}
	\liminf_{R\to\infty} \frac{\log u_t^{*}(R)}{%
	\left( \log R\right)^{2/(4-\alpha)}}>
	\frac{C}{\varkappa^{\alpha/(4-\alpha)}}
	\qquad\text{a.s.},
\end{equation}
where $C$ does not depend on $\varkappa$. For the other bound, we start with a modulus of continuity estimate, viz.,
\begin{equation}
 \|u_t(x)-u_t(y)\|_{2k} \le \text{const} \cdot \|x-y\| + \left(8k \int_0^t 
	\sup_{a\in\R^d}\|u_s(a)\|_{2k}^2 \
	\mathcal{I}_s\,\d s\right)^{1/2},
\end{equation}
where $\mathcal{I}_s:={\iint}_{\R^d\times\R^d}
\d w\,\d z\  |\mathcal{H}(w)\mathcal{H}(z)|f(w-z),$
for $\mathcal{H}(\xi) := p_{t-s}(\xi-x)-p_{t-s}(\xi-y)$
for all $\xi\in\R^d$.
Because of Proposition \ref{pr:pa_mom_Riesz}, we can simplify our
estimate to the following:
\begin{equation}
	\|u_t(x)-u_t(y)\|_{2k}
	\le \text{const} \cdot \|x-y\| + \overline{u}_0\e^{2\overline{c}t (2k)^{2/(2-\alpha)}
	\varkappa^{-\alpha/(2-\alpha)}}
	\left(8k \int_0^t \mathcal{I}_s\,\d s\right)^{1/2}.
\end{equation}
The simple estimate
$\int_{\R^d}|\mathcal{H}(z)| f(w-z)\,\d z \le
2 \sup_{z\in\R^d}(p_{t-s}*f)(z)$,
together with \eqref{eq:p*f} yields
\begin{equation}
	\mathcal{I}_s
	\le \frac{\textnormal{const}}{(t-s)^{\alpha/2}\varkappa^{\alpha/2}}
	\cdot\int_{\R^d} |\mathcal{H}(w)|\,\d w
	\le  \frac{\textnormal{const}}{(t-s)^{\alpha/2}}
	\cdot\left(\frac{\|x-y\|}{(t-s)^{\nicefrac12}}\wedge 1\right),
\end{equation}
where ``const'' does not depend on $(x\,,y\,,s\,,t)$,
but might depend on $\varkappa$;
see Lemma \ref{p_bound} for the last inequality. These remarks, and
some computations together show that, uniformly for
all $x,y\in\R^d$ that satisfy $\|x-y\|\le 1\wedge t^{\nicefrac12}$,
$\E ( |u_t(x)-u_t(y)|^{2k} ) \le C \|x-y\|^{ \varpi k}$,
where  $C:=C(k\,,\varkappa\,,t\,,d\,,\alpha)$ is positive and finite and $\varpi=\min(1, 2-\alpha)$.
Now a quantitative form of the Kolmogorov continuity theorem 
\cite[(39), p.\ 11]{Minicourse}
tells us that uniformly for all hypercubes $T\subset\R^d$ of side length
$\le d^{-1/2} (1\wedge t^{\nicefrac12})$, and for all $\delta\in(0\,,1\wedge t^{\nicefrac12})$,
\begin{equation}
	\E\left(\sup_{\substack{x,y\in T\\ \|x-y\|\le \delta}}
	\left\vert u_t(x)-u_t(y)\right\vert^{2k}\right)
	\le \textnormal{const}\cdot
	(\delta^{\varpi} k)^{ k}
	\exp\left(\frac{\overline{c}t (2k)^{(4-\alpha)/(2-\alpha)}}{%
	\varkappa^{\alpha/(2-\alpha)}}\right),
\end{equation}
where ``const'' depends only on $(\varkappa\,,t\,,d\,,\alpha)$. 
We now split $[0\,,R]^d$ into sub-cubes of sidelength $\textnormal{const}
\cdot (1\wedge t^{\nicefrac12})$, each of 
which is contained in a ball of radius $(1\wedge t^{\nicefrac12})/2$. Let $\mathcal{C}_R$ denote the collection of mentioned subcubes and $\mathcal{M}_R$, the set of midpoints of these subcubes. We can then observe
the following:
\begin{equation}
	\P\left\{ u_t^{*}(R) >2\mathcal{E}_M\right\} 
	\le \sum_{x \in \mathcal{M}_R} \P\left\{
	|u_t(x)|>\mathcal{E}_M \right\}
	+\sum_{T\in \mathcal{C}_R} \P\left\{
	\mathop{\textnormal{Osc}}\limits_T(u_t)>\mathcal{E}_M\right\},
\end{equation}
where $\textnormal{Osc}_T(g):=\sup_{x,y\in T}|g(x)-g(y)|$,
and $c$ depends only on $(t\,,d)$.  
In this way we find that
\begin{equation}
	\P\left\{ u_t^*(R) >2\mathcal{E}_M \right\}
	\le AR^d\times \left[\frac{\e^{Ak^{(4-\alpha)/(2-\alpha)}
	\varkappa^{-\alpha/(2-\alpha)}}}{\e^{Mk(\log R)^{2/(4-\alpha)}
	\varkappa^{-\alpha/(4-\alpha)}}}\right],
\end{equation}
where $A\in(0\,,\infty)$ is a constant that
depends only on $(t\,,\varkappa\,, \alpha\,, d)$.
Finally, we choose $k :=\varkappa^{\alpha/(4-\alpha)} 
(\log R )^{(2-\alpha)/(4-\alpha)}$ and $M$ large enough to ensure that
$\P\{ u^*_t(R)>2\mathcal{E}_M\}= O(R^{-2})$ as $R\to\infty$.
An application of Borel-Cantelli lemma proves the result.
\qed

\subsection*{Acknowledgements}
We thank Professor Firas Rassoul-Agha for providing us with 
reference \cite{KZ} and a number of insightful conversations.

\begin{small}
\bigskip

\noindent\textbf{Daniel Conus},
\noindent Department of Mathematics, Lehigh University,
		Bethlehem, PA 18015. \\
		
\noindent\textbf{Mathew Joseph and Davar Khoshnevisan},\\
\noindent Department of Mathematics, University of Utah,
		Salt Lake City, UT 84112-0090 \\ 
		
\noindent\textbf{Shang-Yuan Shiu}, Institute of Mathematics, Academia Sinica, Taipei 10617\\
		
\noindent\emph{Emails} \& \emph{URLs}:\\
	\indent\texttt{daniel.conus@lehigh.edu},\hfill
		\url{http://www.lehigh.edu/~dac311/}\\
	\indent\texttt{joseph@math.utah.edu}, \hfill
		\url{http://www.math.utah.edu/~joseph/}\\
	\indent\texttt{davar@math.utah.edu}\hfill
		\url{http://www.math.utah.edu/~davar/}\\
	\indent\texttt{shiu@math.sinica.edu.tw}
\end{small}

\end{document}